\documentclass[12pt, twoside]{article}
\usepackage{amsmath,amsthm,amssymb}
\usepackage{times}
\usepackage{enumerate}
\usepackage[english]{babel} \usepackage{inputenc, amsmath, amssymb , latexsym,
  epic, epsfig, rotating,fancyhdr, amsthm, pifont, empheq,color}
  
\def\titlerunning#1{\gdef\titrun{#1}}
\makeatletter
\def\author#1{\gdef\autrun{\def\and{\unskip, }#1}\gdef\@author{#1}}
\def\address#1{{\def\and{\\\hspace*{18pt}}\renewcommand{\thefootnote}{}%
\footnote {#1}}%
\markboth{\autrun}{\titrun}}
\makeatother
\def\email#1{e-mail: #1}
\def\subjclass#1{{\renewcommand{\thefootnote}{}%
\footnote{\emph{Mathematics Subject Classification (2010):} #1}}}

\newcommand{\eqand}{\ensuremath{\quad \textrm{and} \quad}}

\newcommand{\ssq}{\ensuremath{\subseteq}}
\newcommand{\smin}{\ensuremath{\setminus}}
\newcommand{\eps}{\ensuremath{\varepsilon}}

\newcommand{\inte}{\ensuremath{\mathrm{int}}}

\newcommand{\kreis}{\ensuremath{\mathbb{T}^{1}}}

\newcommand{\torus}{\ensuremath{\mathbb{T}^2}}

\newcommand{\homeo}{\ensuremath{\mathrm{Hom}}}

\newcommand{\norm}[1]{\ensuremath{\parallel \! #1 \! \parallel}}
 
\newcommand{\twomatrix}[4]{\ensuremath{\left(\begin{array}{cc} #1 & #2 \\ #3 &
      #4 \end{array}\right)}}
\newcommand{\twovector}[2]{\ensuremath{\left(\begin{array}{c} #1 \\
        #2 \end{array}\right)}}

\newcommand{\alphlist}{\begin{list}{(\alph{enumi})}{\usecounter{enumi}\setlength{\parsep}{2pt}
      \setlength{\itemsep}{1pt} \setlength{\topsep}{5pt}
      \setlength{\partopsep}{3pt}}}
\newcommand{\arablist}{\begin{list}{(\arabic{enumi})}{\usecounter{enumi}\setlength{\parsep}{2pt}
          \setlength{\itemsep}{1pt} \setlength{\topsep}{5pt}
          \setlength{\partopsep}{3pt}}}
\newcommand{\romanlist}{\begin{list}{(\roman{enumi})}{\usecounter{enumi}\setlength{\parsep}{2pt}
              \setlength{\itemsep}{1pt} \setlength{\topsep}{5pt}
              \setlength{\partopsep}{3pt}}}
\newcommand{\Romanlist}{\begin{list}{(\Roman{enumi})}{\usecounter{enumi}\setlength{\parsep}{2pt}
              \setlength{\itemsep}{1pt} \setlength{\topsep}{5pt}
              \setlength{\partopsep}{3pt}}}
\newcommand{\bulletlist}{\begin{list}{$\bullet$}{\setlength{\parsep}{2pt}
                \setlength{\itemsep}{1pt} \setlength{\topsep}{5pt}
                \setlength{\partopsep}{3pt}\setlength{\leftmargin}{15pt}}} 
\newcommand{\Alphlist}{\begin{list}{(\Alph{enumi})}{\usecounter{enumi}\setlength{\parsep}{2pt}
      \setlength{\itemsep}{1pt} \setlength{\topsep}{5pt}
      \setlength{\partopsep}{3pt}}}
 \newcommand{\listend}{\end{list}}

\newcommand{\T}{\ensuremath{\mathbb{T}}}

\newcommand{\N}{\ensuremath{\mathbb{N}}} 
\newcommand{\R}{\ensuremath{\mathbb{R}}}
\newcommand{\Z}{\ensuremath{\mathbb{Z}}}
\newcommand{\Q}{\ensuremath{\mathbb{Q}}}
\newcommand{\C}{\ensuremath{\mathbb{C}}}

\newcommand{\A}{\ensuremath{\mathbb{A}}}

\newcommand{\cA}{\mathcal{A}}

\newcommand{\cC}{\mathcal{C}}

\newcommand{\cE}{\mathcal{E}}

\newcommand{\cN}{\mathcal{N}}

\newcommand{\cQ}{\mathcal{Q}}

\newcommand{\cU}{\mathcal{U}}

\newcommand{\iLim}{\ensuremath{\lim_{i\rightarrow\infty}}}

\newcommand{\halb}{\ensuremath{\frac{1}{2}}}

\newcommand{\stit}{\medskip \noindent\textbf}
\newcommand{\id}{\textrm{id}}
\newcommand{\sm}{\setminus}
\newcommand{\bd}{\partial}
\newcommand{\floor}[1]{\lfloor #1 \rfloor}
\renewcommand{\norm}[1]{\left\| #1 \right\|}

\newcommand{\ham}{\ensuremath{\mathrm{Ham}}}
\newcommand{\hamlifts}{\ensuremath{\widehat{\mathrm{Ham}}}}

\newcommand{\Fab}{\ensuremath{F_{\alpha,\beta}}}

\newcommand{\ie}{i.e.\ }

\newtheorem{definition}{Definition}[section]
\newtheorem{thm}[definition]{Theorem}

\newtheorem{cor}[definition]{Corollary}  
\newtheorem{prop}[definition]{Proposition}

\newtheorem{conj}[definition]{Conjecture}

\newtheoremstyle{tobrem}{3pt}{3pt}{\normalfont}{0pt}{\bfseries}{.}{0.5em}{}
\theoremstyle{tobrem}

\newtheorem{rem}[definition]{Remark}

\numberwithin{equation}{section}

\begin{document}


\baselineskip=17pt


\titlerunning{Onset of diffusion in the kicked Harper model}

\title{On the onset of diffusion\\ in the kicked
    Harper model}

\author{Tobias J\"ager \and Andres Koropecki \and Fabio Armando Tal}

\date{}

\maketitle

\address{T.~J\"ager: Friedrich Schiller University Jena, Institute of Mathematics; \email{tobias.jaeger@uni-jena.de} \and
A.~Koropecki: Universidade Federal Fluminense, Instituto de Matem\'atica e Estat\'\i stica; \email{ak@id.uff.br} \and
F.A.Tal: Universidade de S\~ao Paulo, Instituto de Matem\'atica e Estat\'\i stica, \email{fabiotal@ime.usp.br}
}



\setlength{\abovedisplayskip}{1.0ex}
\setlength{\abovedisplayshortskip}{0.8ex}

\setlength{\belowdisplayskip}{1.0ex}
\setlength{\belowdisplayshortskip}{0.8ex}

\maketitle

\abstract{We study a standard two-parameter family of area-preserving torus
  diffeomorphisms, known in theoretical physics as the {\em kicked Harper
    model}, by a combination of topological arguments and KAM theory. We
  concentrate on the structure of the parameter sets where the rotation set has
  empty and non-empty interior, respectively, and describe their qualitative
  properties and scaling behaviour both for small and large parameters. This
  confirms numerical observations about the onset of diffusion in the physics
  literature. As a byproduct, we obtain the continuity of the rotation set
  within the class of Hamiltonian torus homeomorphisms. }

\section{Introduction}

The {\em Kicked Harper Family} is a parameter family of torus diffeomorphisms $f_{\alpha,\beta}\colon \T^2\to \T^2$ induced by the maps
\begin{equation} \label{e.standardfamily}
  F_{\alpha,\beta} : \R^2 \to \R^2\ , \quad (x,y) \mapsto \left(x+\alpha
  \sin(2\pi(y+\beta \sin(2\pi x))), y+\beta \sin(2\pi x)\right) \ ,
\end{equation}
with parameters $\alpha,\beta\in\R$. It is the composition of a vertical and a
horizontal skew shift: if we let
\begin{eqnarray}
  V_\beta(x,y) & = & (x,y+\beta\sin(2\pi x)) \ ,\\ H_\alpha(x,y) & = &
  (x+\alpha\sin(2\pi y),y) \ ,
\end{eqnarray}
which induce the corresponding maps $v_\beta, h_\beta\colon \T^2\to \T^2$, then
\begin{equation}
  F_{\alpha,\beta} \ = \ H_\alpha\circ V_\beta,\quad  f_{\alpha,\beta} = h_\alpha \circ v_\beta\ . 
\end{equation}
Note that since the shear maps $H_\alpha$ and $V_\beta$ have Jacobian equal to $1$, these maps are area-preserving. This also allows to see the maps $f_{\alpha,\beta}$ are all {\em Hamiltonian
  torus diffeomorphisms}, by which we mean that they are homotopic to the
identity, preserve area and have zero Lebesgue rotation number. As
(\ref{e.standardfamily}) presents one of the simplest ways to produce explicit
examples of Hamiltonian torus diffeomorphisms, one can see it as a standard
family that may serve as a reference for the study of their dynamics and
rotational behaviour. Moreover, this model has been associated to a variety of
problems in theoretical physics, including the motion of magnetic field lines,
wave-particle interactions, dynamics of particle accelerators or laser-plasma
coupling, and both its classical and quantum dynamics have been studied with
computational methods by a variety of authors (see, for example
\cite{HowardHohs1984Reconnection,LeboeufKurchanFeingoldArovas1990PhaseSpaceLocalization,Leboeuf1998KickedHarper,ShinoharaAizawa1997ShearlessKAMBreakup,ShinoharaAizawa1998NontwistMapsTransitionToChaos,Shinohara2002DiffusionThreshold,Zaslavsky2005HamiltonianChaos,Zaslavsky2007PhysicsOfChaos}
and references therein). In the context of KAM theory (\ref{e.standardfamily})
provides a natural example of an area-preserving diffeomorphism of the torus that does not
satisfy a global twist condition. This fact gives rise to a number of
phenomena that have been studied, again in theoretical and computational
physics, under the names of {\em meandering KAM circles}, {\em separatrix
  reconnection} or the appearence of {\em twin chains}
(e.g.\ \cite{HowardHohs1984Reconnection,Leboeuf1998KickedHarper,ShinoharaAizawa1997ShearlessKAMBreakup,Shinohara2002DiffusionThreshold}). 

The purpose of this article is to study this model from the viewpoint of
rotation theory, which provides a rigorous framework for the description of
(some of) the above-mentioned phenomena.  The main topological invariant in this
theory is the {\em rotation set} of a torus homeomorphism $f:\torus\to\torus$,
homotopic to the identity and with lift $F:\R^2\to\R^2$, which has been
introduced by Misiurewicz and Ziemian in
\cite{MisiurewiczZiemian1989RotationSets} as
  \begin{equation}\label{e.rotset-def}
    \rho(F) \ = \ \left\{ \rho \in\R^2 \ \left| \ \exists n_i\nearrow
    \infty,\ z_i\in\R^2: \ \rho=\iLim
    \left(F^n_i(z_i)-z_i\right)/n_i\right\} \right. \ .
  \end{equation}
 It can be shown that the rotation set is compact and convex, and it usually carries dynamical information (see \S\ref{sec.prelim}). However, from the computational viewpoint, finding rotation sets is a delicate problem (see \cite{JaegerPadbergPolotzek2017RotationSets}).  
 As a basis for our further investigations, we first show the continuity of
  the map $(\alpha,\beta)\mapsto \rho(F_{\alpha,\beta})$. This follows from a
  general result on the continuous dependence of rotation set for Hamiltonian
  homeomorphisms (Theorem~\ref{t.continuity} below). The onset of diffusion and
global chaos in (\ref{e.standardfamily}) then corresponds to the appearence of
rotation sets with non-empty interior, and we aim at a better understanding of
this transition by studying the two complementary parameter regions with empty
and non-empty interior of the rotation set, as shown in
Figure~\ref{f.emptyint}. The analysis of these sets is simplified by the fact
that the maps \Fab\ have a number of symmetries, which directly translate into
symmetries of the rotation sets. In particular, the latter are invariant under
the reflexions along the horizontal and vertical axis (see
Section~\ref{StandardFamily}). Combined with the convexity of the rotation set,
this implies the following
\begin{prop} \label{p.rotset_shapes}
  For any $\alpha,\beta\in\R$, the rotation set $\rho(\Fab)$ is either 
  \begin{itemize}
  \item[(i)] reduced to $\{(0,0)\}$;
  \item[(ii)] a non-degenerate segment contained in the horizontal or vertical axis,
    with midpoint at the origin;
  \item[(iii)] a set with non-empty interior.
  \end{itemize}
\end{prop}

Let
\begin{equation}
  \begin{split}
    \cE & = \ \left\{(\alpha,\beta)\in\R^2\mid \inte(\rho(F_{\alpha,\beta})) = \emptyset\right\},\\
    \cN & = \ \left\{(\alpha,\beta)\in\R^2\mid \inte(\rho(F_{\alpha,\beta}))\neq \emptyset\right\},
  \end{split}
\end{equation}
so that $\cN=\R^2\smin \cE$. The following result provides a theoretical basis
for the numerical approximations of these sets in Figure~\ref{f.emptyint}
(explained below) and thereby also provides a justification for similar
computations in \cite{Shinohara2002DiffusionThreshold}. By $|I|$, we denote the
length of an interval $I\ssq\R$. By $\pi_i:\R^2\to \R$ we denote the canonical
projection to the $i$-th coordinate, and we use the same notation for other
product spaces like $\torus$ or $\A=\kreis\times\R$.
\begin{prop}\label{p.diffusion_threshold}
  For $i=1,2$, we have $|\pi_i(\rho(F_{\alpha,\beta}))|>0$ if and only there
  exists some $z\in\R^2$ and $n\in\N$ such that
  $|\pi_i\left(F_{\alpha,\beta}^n(z)-z\right)| \geq 1$.
\end{prop}
Due to the symmetries mentioned above, this immediately entails
\begin{cor}
  We have $\inte(\rho(F))\neq \emptyset$ if and only if there exist
  $z_1,z_2\in\R$ and $n_1, n_2\in \N$ such that  $|\pi_i\left(F_{\alpha,\beta}^{n_i}(z_i)-z_i\right)| \geq 1$
  for $i=1,2$.
\end{cor}
Since this allows us to detect rotation sets with nonempty interior by
  considering the absolute displacement of orbits instead of asymptotic
  averages, this provides a simple numerical procedure to 
  approximate the sets $\cN$ and $\cE$.  In order to obtain
  Figure~\ref{f.emptyint}, for each parameter (pixel), a test point is chosen
  and iterated 4 million times. If the maximal observed displacement is greater
  than 1 in both directions, the pixel is painted white. The process is repeated
  for a large number of test points. The white region can thus be seen as a
  (lower) approximation of $\cN$. In the red region, which corresponds to $\cE$,
  the color scheme corresponds to the maximum of the observed displacements
  between 0 and 1.  
\smallskip

\begin{figure}[h]
\begin{center}
  \includegraphics[height = 8cm]{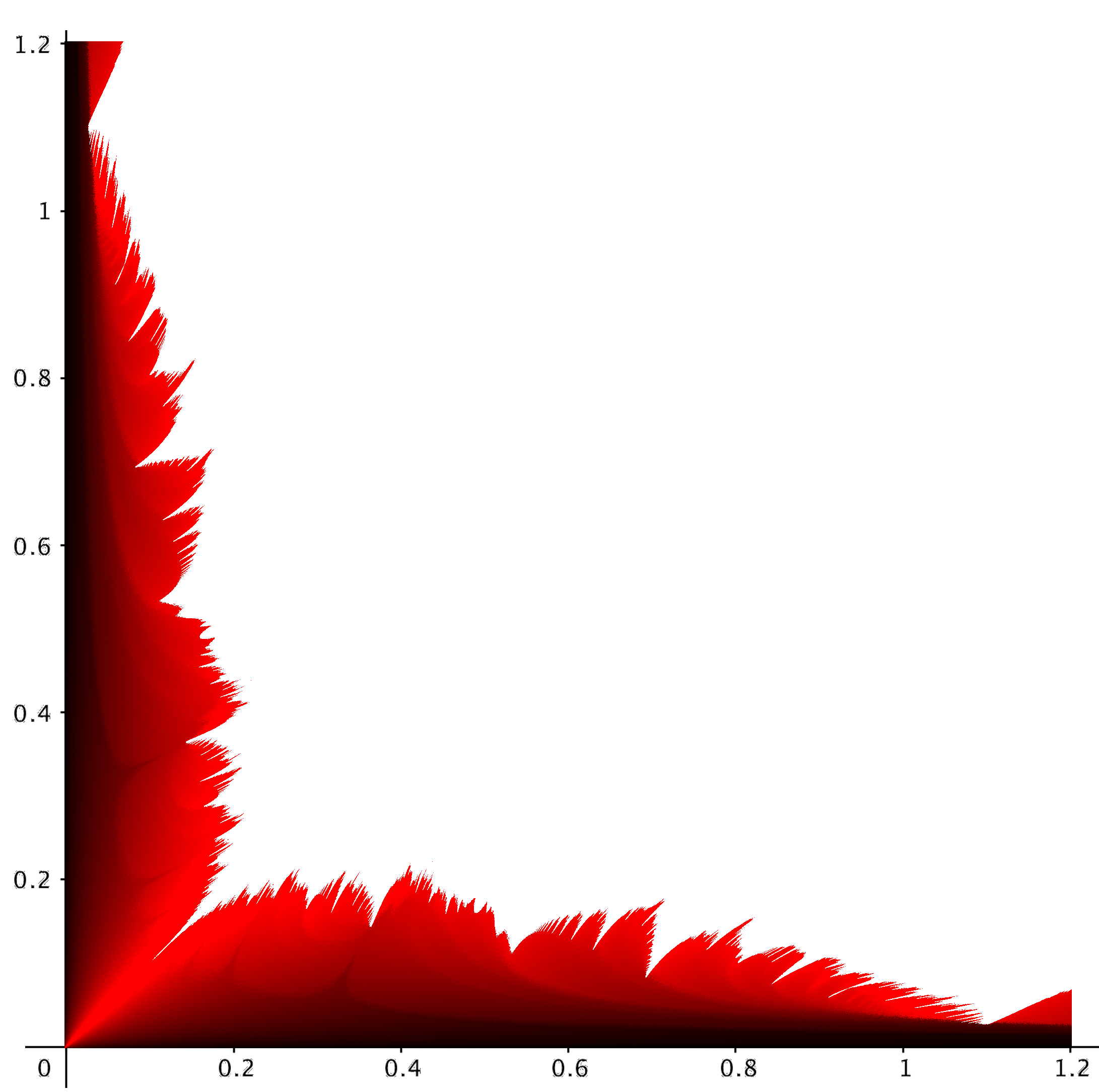}
  \caption{\label{f.emptyint} \small The parameter region $\cE$ on which the
    rotation set of $f_{\alpha,\beta}$ has empty interior is shown in red,
    whereas the white region corresponds to parameters with non-empty interior
    rotation set. The red colour scheme indicates the amount of vertical
    movement (below the diagonal) or horizontal movement (above the
    diagonal). Dark red corresponds to very small displacements, whereas light
    red indicates displacments close to the critical threshold of one.  }
\end{center}
\end{figure}

The fact that the rotation set has empty interior in a neighborhood of the
coordinate axes (removing the origin) is easily explained by the KAM phenomenon, \ie the persistence of certain invariant circles given by Theorem \ref{t.KAM}: small perturbations of integrable twist maps have some invariant circles (so called KAM circles) which are ``continuations'' of certain invariant circles of the twist map (those whose rotation numbers satisfy a given algebraic condition). Indeed, if one parameter of $f_{\alpha,\beta}$ is fixed (and nonzero) and the other is small enough, the dynamics in a neighborhood of one of the axes is a small perturbation of an integrable twist map, so that KAM circles persist and force the boundedness of
orbits in the transverse direction. Hence, for any $\alpha\neq 0$ there exists
$\beta_0>0$ such that $\{(\alpha, \beta) : |\beta|\leq \beta_0\}\subset \cE$. A
quantitative refinement of this statement will be given in
Theorem~\ref{t.scaling} below. In contrast, when both parameters are large, one
can guarantee the creation of \emph{rotational horseshoes}, leading to a
rotation set with nonempty interior. This entails the following

\begin{prop} \label{p.rotset_lowerbound} If $|\alpha|\geq 1/2$ and $|\beta|\geq 1/2$ then $(\alpha, \beta)\in \cN$.
\end{prop} 

The transition between the diffusive and the non-diffusive regime is a subtle
problem that is still poorly understood, even in the classical Chirikov-Taylor
standard family
\cite{chirikov1979UniversalInstability,ChirikovShepelyansky2008ScholarpediaChirikovMap}.
A non-trivial qualitative feature of the sets $\cE$ and $\cN$ that can be
observed in Figure~\ref{f.emptyint} is the fact that a thin cusp of $\cN$ seems
to extend along the diagonal towards the origin. This is confirmed by the
following results.
\begin{thm} \label{t.irrotational}
  We have $\rho(F_{\alpha,\beta})=\{(0,0)\}$ if and only if
  $(\alpha,\beta)=(0,0)$. 
\end{thm}
Due to further symmetries of the rotation set on the diagonal explained in
Section~\ref{StandardFamily}, this directly implies
\begin{cor}\label{c.diagonal}
  We have $\inte(\rho(F_{\alpha,\alpha}))=\emptyset$ if and only if $\alpha=0$.
\end{cor}
Hence, the diagonal is indeed contained in $\cN$. Further, the following
statement confirms the cusp-like shape of $\cN$ near the origin.

\begin{thm} \label{t.pinching}
  Suppose that $\lambda \in [0,1)$. Then
  \[
     \alpha_0(\lambda) \ := \ \inf\{\alpha>0\mid (\alpha,\lambda\alpha)\in \cN\}
  \    = \inf\{\alpha>0\mid (\lambda\alpha,\alpha)\in \cN\} \  > \  0, \
  \] 
	and $\alpha_0$ is uniformly bounded away from zero on any compact subinterval of $[0,1)$.
\end{thm}
 We remark that the middle equality in the statement of the theorem above is justified by the symmetries of the rotation set discussed in
Section~\ref{StandardFamily}.

Turning away from the vicinity of the origin, we then focus on large parameters
near the coordinate axes. Here, Figure~\ref{f.emptyint_along_axes} reveals both
a periodic structure combined with a decay in height of the region $\cE$ as the
parameter $\alpha$ tends to infinity. Both the periodicity and the scaling
behaviour are explained in \cite{Shinohara2002DiffusionThreshold} on a heuristic
level, by rescaling the maps $F_{\alpha,\beta}$ in a suitable neighbourhood of
the critical line $\R\times\{1/4\}$ and relating them to a quadratic
approximation, given by the {\em standard non-twist map} (see
\cite{ShinoharaAizawa1997ShearlessKAMBreakup,ShinoharaAizawa1998NontwistMapsTransitionToChaos}). As
the argument relies on some a~priori assumptions that are hard to verify, it is
difficult to convert it into a rigorous proof. However, we can at least use
these ideas to obtain analytic estimates for the scaling behaviour. Given
$\alpha\in\R$, we let
\begin{equation}\label{e.beta_pm}
  \beta^-(\alpha) \ = \ \inf\{\beta>0\mid (\alpha,\beta)\in\cN\} \eqand
  \beta^+(\alpha) \ = \ \sup\{\beta>0\mid (\alpha,\beta) \in\cE\} \ .
\end{equation}
Then we have
\begin{thm}\label{t.scaling}
  There exists constants $0<c<C$ such that
  \begin{equation}
    \frac{c}{\sqrt{\alpha}} \ \leq \  \beta^-(\alpha) \  \leq \ \beta^+(\alpha) \  \leq \ \frac{C}{\sqrt{\alpha}}
  \end{equation}
  for all $\alpha\geq 1$.
      \end{thm}
\medskip

\begin{figure}[h]
\begin{center}
  \includegraphics[width=\linewidth]{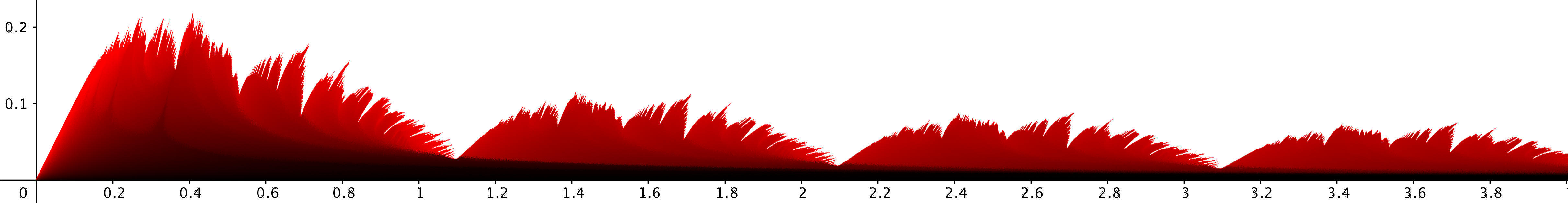}
  \caption{\label{f.emptyint_along_axes} \small The picture shows the part of
    the parameter set $\cE$ that lies below the diagonal, with $\alpha$ between
    0 and 4. It reveals a seemingly periodic structure, combined with a decay in
    the height of the region when $\alpha$ becomes large. }
\end{center}
\end{figure}

The paper is organised as follows. In Section~\ref{StandardFamily}, we collect a
number of basic facts about the kicked Harper map, including a description of
its symmetries, the local analysis of canonical fixed points and a proof of
Proposition~\ref{p.rotset_lowerbound}. Section~\ref{DiffusionThreshold} deals
with the diffusion threshold provided by
Proposition~\ref{p.diffusion_threshold}.  The continuous dependence of the
rotation set of Hamiltonian torus diffeomorphisms is proved in
Section~\ref{RotationSetContinuity}. Section~\ref{Cusp} provides the proofs of
Theorems~\ref{t.irrotational} and \ref{t.pinching} and their corollaries,
describing the cusp of $\cN$ along the diagonal. The proof of
Theorem~\ref{t.scaling} about the scaling behaviour for large parameters is then
given in Section \ref{LargeParameters}.
We conclude with Section~\ref{questionsandremarks}, presenting some additional remarks on the family as well as a few interesting questions for further work.
\bigskip

\noindent{\bf Acknowledgments.} We would like to thank the anonymous referees for the many suggestions and corrections that helped improve this paper. Most of this research was carried out while AK
and FT spent a year at the Friedrich Schiller University of Jena as visiting
professors. The stay of AK was supported by a Mercator Fellowship of the German
Research Council \mbox{(DFG-grant OE 538/9-1)}. The stay of FT was made possible
by a Willhelm Friedrich Bessel Award of the Alexander von Humboldt
Foundation. TJ acknowlegdes support by a Heisenberg grant of the German Research
Council (DFG-grant OE 538/4-1).

\section{Preliminaries}\label{sec.prelim}
We denote by $\homeo(X)$ the space of homeomorphisms of a topological space $X$. The torus $\T^2$ is regarded as the quotient map $\R^2/\Z^2$ and we denote by $\pi\colon \R^2\to \T^2$ the projection, which is a universal covering map. We also denote by $\widehat{\homeo}(\T^2)$ the space of all maps $F\colon \R^2\to \R^2$ which are lifts of some element of $\homeo(\T^2)$. This corresponds to all homeomorphisms of the form $A + \Delta$ where $A\in \mathrm{GL}(2,\Z)$ and $\Delta$ is $\Z^2$-periodic. We remark that if $f\in \homeo(\T^2)$ is the map lifted by $F$, then the element $[F]\in  \mathrm{GL}(2,\Z)$ such that $F-[F]$ is $\Z^2$-periodic depends only on $f$ and not on the choice of the lift, so we may also denote it by $[f]$. If one identifies the first homology group of $\T^2$ with $\Z^2$, then $[f]$ represents the homomorphism induced in first homology by $f$. 

When $[f] = \id$ one has that $f$ is homotopic to the identity. We denote the space of homeomorphisms homotopic to the identity by $\homeo_0(\T^2)$, and its corresponding lifts by $\widehat \homeo_0(\T^2)$.  

\subsection{Rotation sets and vectors}
Fix $f\in \homeo_0(\T^2)$ and a lift $F$ of $f$. 
The map $\hat \Delta_F := F-\id$ is $\Z^2$-periodic and therefore induces a continuous map $\Delta_F \colon \T^2\to \R^2$. One may easily verify that $\Delta_{F^n}(z) = \sum_{k=0}^{n-1} \Delta_F(f^k(x))$. 
Recall the definition of the rotation set $\rho(F)$ from  (\ref{e.rotset-def}). Equivalently, $\rho(F)$ is the set of all limits of sequences of the form
\begin{equation}\label{e.rotset_birkhoff}
\frac{1}{n_i} \Delta_{F^{n_i}}(z_i) = \frac{1}{n_i}\sum_{k=0}^{n_i-1} \Delta_F(f^k(z_i)),
\end{equation}
where $n_i\to \infty$ and $z_i\in \T^2$.

 Let us state some general properties of rotation sets (we refer the reader to \cite{MisiurewiczZiemian1989RotationSets} for details):
\begin{prop}\label{p.rotset}
	The following properties hold: 
\begin{itemize}
	\item[(1)] For all $v\in \Z^2$ and $n\in \Z$, one has $\rho(F^n+v) = n\rho(F)+v$;
	\item[(2)] For any $H\in \homeo(\T^2)$, one has $\rho(HFH^{-1}) = [H]\rho(F)$.
	\item[(3)] The rotation set is compact and convex.
\end{itemize}
\end{prop}

In general $\rho(F)$ does not depend continuously on $F\in \widehat\homeo_0(\T^2)$, however:
\begin{prop}[\cite{MisiurewiczZiemian1989RotationSets}]\label{p.rotset_semi}
	The map $\rho:F\mapsto \rho(F)$ is upper semicontinuous in the Hausdorff topology.
\end{prop}
As a consequence, if $\rho(F)$ is a singleton then $\rho$ is continuous at $F$. We also have:
\begin{prop}[\cite{misiurewicz/ziemian:1991}]\label{p.rotset_cont_int}
	 If $\rho(F)$ has nonempty interior, then the function $\rho$ is continuous at $F$.
\end{prop}

The rotation vector of a point $z\in \R^2$ is
$$\rho(F,z) = \lim_{n\to \infty} \frac{F^n(z)-z}{n},$$ if this limit exists, and since $F^n(z)-z$ is $\Z^2$-periodic we may define the rotation vector of $z\in \T^2$ as $\rho(F, z) = \rho(F, z')$ where $z'\in \pi^{-1}(z)$, if it exists. Note that in this case $\rho(F,z)\in \rho(F)$. In general not every $v\in \rho(F)$ is the rotation vector of some point (as there are examples where this fails to hold when $\rho(F)$ has empty interior), but this is true if $v$ is either an interior point or an extremal point of $\rho(F)$ (see \cite{misiurewicz/ziemian:1991}).

The sequences in (\ref{e.rotset_birkhoff}) are Birkhoff averages for the map $\Delta_F$. From Birkhoff's Ergodic Theorem, if $\mu$ is an $f$-invariant Borel probability measure, one sees that the $\rho(F,z)$ exists for $\mu$-almost every $z\in \T^2$ and $$\int \rho(F,z)d\mu(z) = \int \Delta_F(z) d\mu(z).$$
The number $\rho_\mu(F) = \int \Delta_F d\mu$ is called the \emph{mean rotation vector} of $F$ for the measure $\mu$, and due to the convexity of the rotation set one always has $\rho_\mu(F) \in \rho(F)$. 
If $\mu$ is ergodic, the Ergodic Theorem also guarantees that $\rho(F,z) = \rho_\mu(F)$ for $\mu$-almost every $z$. Moreover, it is known that every element of $\rho(F)$ that is either extremal or interior is the mean rotation vector of some ergodic probability \cite{misiurewicz/ziemian:1991}.

Every periodic point of $f$ has a well-defined rotation vector, which belongs to $\mathbb{Q}^2$. In fact, a periodic point of $f$ of period $q$ lifts to a point $z\in \R^2$ such that $F^q(z) = z+v$ for some $v\in \Z^2$, and one has $\rho(F,z) = v/q$. The converse of this observation is partially true:
\begin{prop}[Realization by periodic points]\label{p.reali}
	If $w = v/q \in \rho(F)$ with $v\in\Z^2$ and $q\in \N$, then there exists $z\in \R^2$ such that $F^q(z) = z+v$, provided that one of the following properties holds:
	\begin{enumerate}
		\item[(1)] $w$ is an extremal point of $\rho(F)$ \cite{franks:1988a};
		\item[(2)] $w$ is an interior point of $\rho(F)$ \cite{franks:1989};
		\item[(3)] $F$ is area-preserving and $\rho(F)$ is an interval \cite{franks:1995};
		\item[(4)] $F$ is area-preserving and $w$ belongs to the convex hull of $\rho(F)\cup B_\epsilon(\rho_\lambda(F))$ for every $\epsilon>0$, where $\rho_\lambda(F)$ denotes the mean rotation vector associated to Lebesgue measure. \cite[Prop. 2.1]{franks:1995}. 
	\end{enumerate}	
\end{prop}

Further, it is known that when $\rho(F)$ has nonempty interior, $f$ has positive topological entropy \cite{llibre/mackay:1991}.

\begin{rem}
For the particular case of the Kicked Harper model, the second item of Proposition \ref{p.reali} together with the symmetries of the problem explored in Subsection \ref{sec:sym} show that,  if $\rho(F_{\alpha,\beta})$ has nonempty interior, then for every direction in $\Z^2$ one can find a periodic orbit of  $f_{\alpha,\beta}$ that lifts to a point moving with linear speed in that direction. This is the reason we refer to this situation as ``diffusive''. 
\end{rem}

\subsection{Area-preserving and Hamiltonian homeomorphisms}
	Let $\lambda$ denote the Lebesgue measure in $\R^2$, which induces a measure (also denoted $\lambda$) on $\T^2$. 
	We denote by $\homeo_0^{\textrm{ap}}(\T^2)$ the space of area-preserving elements of $\homeo_0(\T^2)$, which are those preserving the Lebesgue measure. If $F$ is a lift of $f\in \homeo_0^{\textrm{ap}}(\T^2)$, denote the mean rotation vector for $\lambda$ by $\rho_\lambda(F)$.
	
	If $\rho_\lambda(F) =(0,0)$ for some lift $F$ of $f$, then $f$ is called a \emph{Hamiltonian homeomorphism}, and the lift $F$ is its Hamiltonian lift. 	We denote by $\ham(\T^2)$ the space of Hamiltonian diffeomorphisms of $\T^2$ and $\widehat{\ham}(\T^2)$ the space of the respective Hamiltonian lifts. 
	The reason for this nomenclature comes from the analogous notion for diffeomorphisms: a Hamiltonian diffeomorphism is one which is the time-$1$ map of the flow associated to a time-dependent Hamiltonian vector field. Note that area-preserving diffeomorphisms of $\T^2$ are the same as symplectic diffeomorphisms. It is known that the symplectic diffeomorphisms of $\T^2$ homotopic to the identity which are Hamiltonian are precisely those with a lift whose mean rotation vector is $(0,0)$ (see \cite{franks-handel,oh2015symplectic}); thus by extension one says that a element of $\homeo_0^{\textrm{ap}}(\T^2)$ is Hamiltonian if it has a lift with mean rotation vector $(0,0)$.
	
	Note that we have
	\[
	\ham(\torus) \  \ssq
	\ {\mathrm{Homeo}}_0^{\mathrm{ap}}(\torus) \ \ssq \ {\mathrm{Homeo}}_0(\torus) \ ,
	\]
	Moreover, Proposition \ref{p.reali} guarantees that a Hamiltonian lift always has a fixed point. We will need a more detailed result about Hamiltonian lifts, which can be found in \cite[Theorem 70]{LeCalvezTal2015ForcingTheory}:
	\begin{thm}\label{t.hamiltonian_LCT}
		Suppose the fixed point set of $f\in \ham(\T^2)$ is contained in a topological disk, and let $F$ be its Hamiltonian lift. Then one of the following holds:
		\begin{itemize}
			\item[(1)] $\rho(F)$ has nonempty interior,  and the origin lies in its interior;
			\item[(2)] $\rho(F) =  \{tu\mid a\leq t \leq b\}$, where $u\in\Z^2$ and $a < 0 < b$;
			\item[(3)] $\rho(F)=\{(0,0)\}$, and $\sup\{\|F^n(z)-z\| : n\in \Z, z\in \R^2\} < \infty$. 
		\end{itemize}
	\end{thm}
	When $\rho(F) = \{(0,0)\}$, we say that $F$ is irrotational. Thus, part (3) says that an irrotational Hamiltonian lift has uniformly bounded displacements.
	Finally, we need the following simple result (see, for instance, \cite[Proposition 6.1]{KoropeckiTal2012StrictlyToral})
	\begin{prop}\label{p.fix-essential}
	 	Suppose the fixed point set of $f\in \homeo_0^{\textrm{ap}}(\T^2)$  is not contained in a topological disk. Then the rotation set of a lift of $f$ is either a line segment with rational slope or a singleton.
	\end{prop}

\subsection{The invariant curve theorem}\label{s.KAM}
We will use the classical invariant curve theorem from KAM theory, which is originally due to Moser \cite{moser-KAM}. First let us recall that a diffeomorphism $f\colon \T^1\times \R\to \T^1\times \R$ is called \emph{exact symplectic} if it preserves the Lebesgue area form on $\T^1\times \R$ and the differential form $f^*(y\,dx) - y\,dx$ is exact. This means that the algebraic area between a simple loop $C\subset \T^1\times \R$ not homotopic to a point and its image $f(C)$ is zero. We state the version of the invariant curve theorem that we will use (which can be found in a more general form in \cite{bost}, for example).

\begin{thm}\label{t.KAM}
Let $f\colon \T^1\times \R\to \T^1\times \R$ a symplectic diffeomorphism of the form $f(x,y)=(x+\omega(y), y)$, where $\omega\colon \R^1\to \T^1$ is a $\cC^\infty$ map.  Suppose that $y_0\in \R$ is such that: 
\begin{itemize}
	\item[(i)] $\alpha = \omega(y_0)$ satisfies the following Diophantine condition: there exist $\tau>0$ and $K>0$ such that $\big|q\alpha - p\big| \geq K|q|^{-\tau}$ for all $(p,q)\in \Z^2\sm\{0\}$;
	\item[(ii)] $\frac{d\omega}{dy}(y_0)\neq 0$.
\end{itemize}
Then, any $\cC^\infty$ exact symplectic diffeomorphism $g$ sufficiently close to $f$ in the $\cC^\infty$ topology has an invariant circle $C_g$ which is the graph of a $\cC^\infty$ map $u_g\colon \T^1\to \R$, such that $g|_{C_g}$ is topologically conjugate to $f|_{\T^1\times \{y_0\}}$. Moreover, $u_g$ varies continuously with $g$.
\end{thm}

We remark that, as it is well known, the real numbers satisfying a Diophantine condition as above are dense in $\R$ (moreover, they have full Lebesgue measure). In particular, even if (i) does not hold, condition (ii) implies that arbitrarily close to $y_0$ there are values of $y$ for which it holds, and therefore there are invariant circles which are persistent by $\cC^\infty$-small perturbations. We refer to the continuations of these invariant circles as \emph{KAM circles}.

\section{Basic observations on the kicked Harper map}\label{StandardFamily}

The aim of this section is to collect a number of basic observations about the
symmetries and the fixed and periodic points of the maps $f_{\alpha,\beta}$
given by (\ref{e.standardfamily}) and their lifts \Fab. In order to simplify
notation, we let
\[
s(x) \ = \ \sin(2\pi x) \ ,
\]
so that
\[
F_{\alpha,\beta}(x,y) \ = \ (x+\alpha s(y+\beta s(x)),y+\beta s(x)) \ 
\]
and
\[
H_\alpha(x,y) \ = \ (x+\alpha s(y),y),\quad  V_\beta(x,y) \ = \ (x,y+\beta s(x)) \ .
\]
Then we have $\Fab=H_\alpha\circ V_\beta$, which allows to see that \Fab\ is
area-preserving and also extends to a biholomorphic diffeomorphism of $\C^2$, which will be important in Section \ref{sec:irrotational}.

\subsection{Symmetries}\label{sec:sym}

The maps $f_{\alpha,\beta}$ and their lifts $F_{\alpha,\beta}$ have a number of
natural symmetries, which directly translate to symmetries of their rotation
sets $\rho(F_{\alpha,\beta})$ and the map
$(\alpha,\beta)\mapsto(\rho(F_{\alpha,\beta}))$. 

To properly state the symmetries, consider the following linear involutions:
\begin{align*}
S_1\colon (x,y)\mapsto& (-x, y),& S_2\colon (x,y)&\mapsto(x,-y),\\
S\colon (x,y)\mapsto& (-x,-y), &D\colon (x,y)&\mapsto(y,x).
\end{align*}
We will also use the general fact that $\rho(F^{-1}) = S\rho(F)$.  Futher, note
that
\begin{equation}
\label{eq:sym-diag}
H_\alpha\circ D = D\circ V_\alpha\quad \text{ and } \quad V_\beta\circ D = D\circ H_\beta. 
\end{equation}
which, noting that $H_\alpha^{-1}=H_{-\alpha}$ and $V_{\beta}^{-1}=V_{-\beta}$, implies
\begin{equation}
\label{eq:sym-diag-F}
\Fab\circ D = D\circ F_{-\beta,-\alpha}^{-1}
\end{equation}

Moreover,
\begin{equation}
\label{eq:sym-sym}
H_\alpha\circ S_i = S_i\circ H_{-\alpha}\quad \text{ and } \quad V_{\beta}\circ S_i = S_i\circ V_{-\beta}  \quad \text{ for $i\in \{1,2\}$,}
\end{equation}
which implies
\begin{equation}
\label{eq:reflection-sym}
F_{\alpha,\beta}\circ S_i = S_i\circ F_{-\alpha, -\beta}\quad \text{ for $i\in \{1,2\}$},
\end{equation}
and, since $S=S_1\circ S_2$, we also have
\begin{equation}
\label{eq:reflection-origin}
F_{\alpha,\beta}\circ S = S\circ F_{\alpha,\beta}.
\end{equation}

\stit{Reversibility.} For any $G\in \{H_\alpha\circ S_1,\, H_\alpha \circ S_2,\, S_1 \circ V_\beta,\, S_2\circ V_\beta\}$, one has 
\begin{equation} \label{eq:reversible}
 F_{\alpha,\beta}\circ G = G\circ F_{\alpha,\beta}^{-1},
\end{equation}
and $G^2 = \id$, as one can directly verify using (\ref{eq:sym-sym}). This
property of a map being conjugated to its inverse by means of an involution is
often referred to as \emph{reversibility} of the dynamics\footnote{The notion of reversibility comes as a generalization of the concept from classical mechanics of reversible mechanical systems, which are those whose Hamiltonian assumes a particularly simple form $H=K+V$ (kinetic energy + potential). Its consequences go beyond the scope of this article; we refer to \cite{devaney} for further details.}
\medskip

\stit{Reflection symmetries.} First note that from (\ref{eq:reflection-origin})
and Proposition \ref{p.rotset}(2) we see that
\begin{equation}
\label{eq:sym-origin}
\rho(\Fab) = S\rho(\Fab) = \rho(\Fab^{-1}).
\end{equation}
On the other hand, since $H_{-\alpha} = H_\alpha^{-1}$ and
$V_{-\beta}=V_\beta^{-1}$, one easily verifies that
\[
H_\alpha\circ F_{-\alpha,-\beta} = F_{\alpha,\beta}^{-1}\circ H_\alpha. 
\]
In particular, since $[H_\alpha] = \id$, the above equation and Proposition \ref{p.rotset}(2) imply
\begin{equation}
\label{eq:sym-odd}
\rho(F_{-\alpha,-\beta}) = \rho(F_{\alpha,\beta}^{-1}) \stackrel{(\ref{eq:sym-origin})}{=} \rho(\Fab).
\end{equation}
From (\ref{eq:reflection-sym}) and Proposition \ref{p.rotset}(2) we se that
$\rho(\Fab) = S_i \rho(F_{-\alpha,-\beta})$ for $i\in \{1,2\}$, so the above
implies
\begin{equation}
\label{eq:sym-axes}
\rho(\Fab) = S_i \rho(\Fab)\quad \text{ for } i\in \{1,2\}.
\end{equation}
In other words, the rotation set is symmetric with respect to the two coordinate axes.
\begin{rem}
Together with the convexity of the rotation set and the fact that $\rho(\Fab)$
always contains the origin (since the origin is a fixed point), the symmetry
with respect to the two coordinate axes implies
Proposition~\ref{p.rotset_shapes}.
\end{rem}
We also have, using (\ref{eq:sym-diag-F}) and Proposition \ref{p.rotset}(2),
\begin{equation}
\label{eq:sym-diag-rot}
\rho(F_{\alpha,\beta}) =
D\rho(F_{-\beta, -\alpha}^{-1}) \stackrel{(\ref{eq:sym-origin})}{=}
D\rho(F_{-\beta,-\alpha}) \stackrel{(\ref{eq:sym-odd})}{=}
D\rho(F_{\beta,\alpha}).\
\end{equation}

\stit{Rotation symmetry.} Consider the rotation $R:(x,y) \mapsto (-y, x)$ by
$\pi/2$.  Noting that $R = S_1D$, we see that
\begin{equation}
\label{eq:sym-R}
\rho(\Fab) \stackrel{(\ref{eq:sym-axes})}{=} S_1\rho(F_{\alpha, \beta}) \stackrel{(\ref{eq:sym-diag-rot})}{=} S_1D\rho(F_{\beta, \alpha}) = R\rho(F_{\beta,\alpha}).
\end{equation}
In particular, $\rho(F_{\alpha,\alpha}) = R\rho(F_{\alpha, \alpha})$ so that for
parameters in the diagonal $\alpha=\beta$ the rotation set is invariant under
rotations by angle $\pi/2$.

\begin{rem}
  The above symmetries, in particular (\ref{eq:sym-diag-rot}) and
  (\ref{eq:sym-R}), imply that the set $\cN$ is symmetric with respect to the
  diagonal and to rotations by $\pi/2$ around the origin, which allows us to
  restrict our attention to parameters $0\leq \beta\leq \alpha$ below the
  diagonal in order to analyze the structure of the parameter sets $\cN$ and
  $\cE$.
\end{rem}

\begin{rem} The previous analysis relies only on the fact that $s$ is an odd function. Therefore, it also applies if one replaces $s$ by any $1$-periodic odd continuous function.
\end{rem}

\stit{Translation symmetries.} From the fact that $s(x+1/2) = -s(x)$, we obtain
some additional symmetries. Consider the translations
\[ 
  T_1\colon (x,y) \mapsto (x+1/2, y), \quad T_2\colon(x,y)\mapsto (x, y+1/2)
\]
Then it is easily checked that
\begin{equation} \label{e.horizontal_translation_symmetry}
  F_{\alpha,\beta}\circ T_1 = T_1 \circ F_{\alpha,-\beta} \quad \text{ and }\quad
  F_{\alpha,\beta}\circ T_2 = T_2\circ F_{-\alpha,\beta}.
\end{equation}
Consequently Proposition \ref{p.rotset}(2) implies
\begin{equation}
  \label{e.rotset_translation_symmetry}
  \rho(F_{\alpha,\beta}) \ = \ \rho(F_{\alpha,-\beta}) \ =
  \ \rho(F_{-\alpha,\beta}) \ = \ \rho(F_{-\alpha,-\beta}) \ .
\end{equation}
Note that the last equality can be derived from the first one, applied to
$F_{-\alpha,\beta}$, or similarly from the fact that if we let $T=T_1\circ T_2$
then we have
\begin{equation}\label{e.translation_symmetry} F_{\alpha,\beta}\circ T\ =\ T\circ F_{-\alpha,-\beta}
\end{equation}
by (\ref{e.horizontal_translation_symmetry}).  We also note that these facts
only depend on the fact that $s(x+1/2)=-s(x)$ and will remain true if $s$ is
replaced by any other function with this property.\medskip

\subsection{A threshold for diffusion} \label{DiffusionThreshold}

As we saw in Proposition \ref{p.rotset_shapes}, if $\rho(F_{\alpha,\beta})$ has
empty interior it is contained in one of the coordinate axes. It is known that
when the rotation set is a nondegenerate interval, the displacement in the
direction perpendicular to the interval is uniformly bounded (see
\cite{Davalos2013SublinearDiffusion,GuelmanKoropeckiTal2012Annularity}). In our
case, the reversibility of the dynamics allows us to obtain a direct proof and
an explicit bound.

\begin{prop}\label{prop:axis-fb} 
If $z\in \R\times \{0\}$ and $\Fab^{n}(z)\in \R\times \{k/2\}$ for some $k, n\in
\Z$, then $\Fab^{2n}(z) = z + (0,k)$. In particular, $(0, \frac{k}{2n})\in
\rho(\Fab)$. An analogous property holds in the horizontal direction.

%
%
\end{prop}

\begin{proof}
Suppose $z = (t, 0)$ and $\Fab^{n}(z) = (t',k/2)$.  Letting $G = H_\alpha\circ
S_2$, one has $G(z) = z$, so by the reversibility equation (\ref{eq:reversible})
we deduce
$$G(\Fab^{n}(z)) = \Fab^{-n}(G(z)) = \Fab^{-n}(z),$$
and noting that $s(-k/2)=0$ we see that
$$G(\Fab^{n}(z)) = H_{\alpha}(S_2(t', k/2)) = (t', -k/2) = \Fab^{n}(z) -
(0,k).$$ Therefore $\Fab^{-n}(z) = \Fab^n(z) - (0,k)$, which yields $z =
\Fab^{2n}(z) - (0,k)$, and the claim follows.  The analogous claim for the
horizontal direction is proven similarly using $G = S_1 \circ V_\beta$.
\end{proof}

\begin{cor} If $\pi_i(\rho(\Fab)) = \{0\}$, then $|\pi_i(\Fab^n(z) - z)| < 1$ for all $z\in \R^2$ and $n\in \Z$. 
\end{cor}
\begin{proof} 
We consider the case $i=2$, the case $i=1$ is analogous. Assuming that there
exist $z_0\in \R^2$ and $n\in \Z$ such that $|\pi_2(\Fab^n(z_0) - z_0)| \geq 1$,
we need to prove that $\pi_2(\rho(\Fab))\neq \{0\}$.

Since $z\mapsto \pi_2(\Fab^n(z)-z)$ attains the value $0$, by the intermediate
value theorem we may choose $z_0$ such that $|\pi_2(\Fab^n(z_0)-z_0)| =
1$. Moreover, since $\Fab(-z) = -\Fab(z)$ we may assume $\pi_2( \Fab^n(z_0) -
z_0) = 1$ (replacing $z_0$ by $-z_0$ if necessary). In addition, we may assume
that $z_0\in \R\times (-1,0]$ since $z_0$ can be replaced by any of its integer
  translates.

Consider first the case where $z_0\in \R\times (-1/2,0]$, so that
  $\Fab^n(z_0)\in \R\times (1/2,1]$. Then the image by $\Fab^n$ of the
    half-plane $H_0 = \{(x,y):y\leq 0\}$ is bounded from above and intersects
    $\{(x,y): y > 1/2\}$, which implies that its boundary $\bd \Fab^n(H_0) =
    \Fab^n(\R\times \{0\})$ also intersects $\{(x,y): y > 1/2\}$. Since the line
    $\R\times \{0\}$ contains fixed points of $\Fab$, we see that
    $\Fab^n(\R\times \{0\})$ also intersects $\{(x,y):y<1/2\}$, and therefore it
    intersects the line $\R\times \{1/2\}$.  Thus the previous proposition
    implies that $\pi_2(\rho(\Fab))\neq \{0\}$.

In the case that $z_0\in \R\times (-1, -1/2]$, an analogous argument shows that
  $\Fab^n(\R\times \{-1/2\})$ intersects $\R\times \{0\}$, and the previous
  proposition again implies $\pi_2(\rho(\Fab))\neq \{0\}$, completing the proof.
\end{proof}

\subsection{Local analysis for irrotational fixed points}
For parameters $\alpha,\beta\neq 0$, the map $\Fab$ has, up to integer
translations, exactly four fixed points: $(0,0),(0,1/2),(1/2,0)$ and
$(1/2,1/2)$.  In view of the symmetries described in \ref{sec:sym}, we analyze the stability of these fixed points when $\alpha>0$ and $\beta >0$, since the other cases can be easily deduced from these. Note that the Jacobian of
\Fab\ is given by
\begin{equation}\label{eq:jacob}
  DF_{\alpha,\beta}(x,y) \ = \ \twomatrix{4\pi^2\alpha\beta\cos(2\pi
    y'))\cos(2\pi x) + 1}{2\pi\alpha\cos(2\pi y'))}{2\pi\beta\cos(2\pi x)}{1}
\end{equation}
where $y'=y + b\sin(2\pi x)$. At the origin, this simplifies to 
\begin{equation}
  DF_{\alpha,\beta}(0,0) \ = \ \twomatrix{4\pi^2\alpha\beta+1}{2\pi\alpha}{2\pi\beta}{1} \ .
\end{equation}
Its eigenvalues are 
\begin{eqnarray}
  \lambda_1^{(0,0)} & = & 2\pi^2\alpha\beta - 2\pi(\alpha\beta(\pi^2\alpha\beta + 1))^{1/2} + 1 \ < 1 \ ,\\
  \lambda_2^{(0,0)} & = &  2\pi^2\alpha\beta + 2\pi(\alpha\beta(\pi^2\alpha\beta+ 1))^{1/2} + 1 \ > 1 \ ,
\end{eqnarray}
with eigenvectors
\begin{eqnarray}
  v_1^{(0,0)} & = & \twovector{-((\alpha\beta(\pi^2\alpha\beta + 1))^{1/2} -
    \pi\alpha\beta)/\beta}{1} \quad\textrm{and} \\ v_2^{(0,0)} & = &
  \twovector{((\alpha\beta(\pi^2\alpha\beta + 1))^{1/2} +
    \pi\alpha\beta)/\beta}{1}
\end{eqnarray}

In $(1/2,1/2)$ the Jacobian is 
\begin{equation}
  DF_{\alpha,\beta}(0,0) \ = \ \twomatrix{4\pi^2\alpha\beta+1}{-2\pi\alpha}{-2\pi\beta}{1} \ .
\end{equation}
It has the same eigenvalues, $\lambda^{(\halb,\halb)}_1=\lambda^{(0,0)}_1$ and
$\lambda^{(\halb,\halb)}_2=\lambda^{(0,0)}_2$, but with eigenvectors
\begin{eqnarray}
  v_1^{(\halb,\halb)} & = & \twovector{((\alpha\beta(\pi^2\alpha\beta +
    1))^{1/2} - \pi\alpha\beta)/\beta}{1} \quad\textrm{and}
  \\ v_2^{(\halb,\halb)} & = & \twovector{-((\alpha\beta(\pi^2\alpha\beta +
    1))^{1/2} + \pi\alpha\beta)/\beta}{1}
\end{eqnarray}

Thus $(0,0)$ and $(1/2, 1/2)$ are hyperbolic fixed points. The remaining two
fixed points are $(0,1/2)$ and $(1/2, 0)$. The Jacobian at those points is
\begin{equation}
  \twomatrix{1-4\pi^2\alpha\beta}{\pm 2\pi\alpha}{\mp 2\pi\beta}{1} \ .
\end{equation}
Note that its trace is $2-4\pi^2\alpha\beta$ which can only be equal to $2$ if
either $\alpha$ or $\beta$ is $0$. In particular the fixed points are
elementary (an elementary fixed point is one where the Jacobian of the map does not have $1$ as an eigenvalue).

\subsection{Periodic orbits and a priori lower bounds on the rotation set}\label{sec:apriori}
Suppose that $\alpha,\beta\geq n$ for some $n\in\N$. Choose $x,y\in[0,1]$ with
$s(x)=n/\beta$ and $s(y)=n/\alpha$. Then it is easily checked that
\[
\Fab(\pm x,\pm y) \ = \ (x\pm n,y\pm n)
\]
so that $\{(n,n),(-n,n),(n,-n),(-n,-n)\}\ssq \rho(\Fab)$. Hence, by convexity of
the rotation set $[-n,n]^2\ssq \rho(\Fab)$. In the particular case
$\alpha=\beta=n$, we even obtain the equality $\rho(F_{n,n})=[-n,n]^2$, as the
maximal displacement in the $x$- and $y$-direction is exactly $n$.

If $\alpha\geq 1/2$, we can choose $y\in [0,1]$ such that $s(y) = 1/(2\alpha)$, so we have $\Fab(0,y) = (1/2, y)$, and $F_{\alpha,\beta}^2(0,y) = (1,y)$. This implies that $\rho(F_{\alpha,\beta})$ contains $(1/2,0)$ (and by symmetry $(-1/2, 0)$ as well). By a similar argument, if $\beta \geq 1/2$ then $(0,-1/2)$ and $(0,1/2)$ belong to $\rho(F_{\alpha,\beta})$.

In particular, if $\alpha,\beta\geq 1/2$, the rotation set contains
the square $$\cQ=\{(x,y)\in\R^2\mid |x|+|y|\leq 1/2\},$$ so it has non-empty interior. This proves Proposition~\ref{p.rotset_lowerbound}.  As a consequence, this means that we can
restrict to parameters in the $1/2$-neighbourhood of the coordinate axes when
analyzing the sets $\cN$ and $\cE$. 

\subsubsection{Explicit cases of mode-locking}
Although computing the rotation sets explicitly seems to be a difficult problem (see Section \ref{questionsandremarks}), it is possible for some particular parameters. Moreover, one may verify that for some very special parameters, the rotation set has nonempty interior and mode-locks (meaning that it remains constant under small variations of the parameters). Particularly,

\begin{itemize}
	\item[(i)] There exists $\delta>0$ such that $\rho(\Fab) = [-1,1]^2$ for all $(\alpha,\beta)\in [1, 1+\delta]^2$; 
	\item[(ii)] There exists $\delta>0$ such that for $(\alpha, \beta)\in [0.5, 0.5+\delta]^2$, the rotation set $\rho(\Fab)$ is the square $Q$ with vertices $(-0.5, 0)$, $(0, -0.5)$, $(0.5, 0)$, and $(0,0.5)$.
\end{itemize}

To verify this one may use the following simple fact:
\begin{prop} Suppose $v\in \R^2\sm \{0\}$, $u\in \Z^2$ and $c\in \R$ are such that, for all $z$ with $\langle z, v\rangle = c$, one has $\langle F(z)-z, v\rangle \leq \langle u, v\rangle$. Then $\rho(F)\subset \{w : \langle w, v\rangle \leq \langle u, v\rangle\}$.
\end{prop}
\begin{proof}
	Let $H_0 = \{z : \langle z, v\rangle \leq c\}$. Note that $H_0$ is bounded by the line $L = \{z : \langle z, v\rangle = c\}$. The hypothesis implies that for $z\in L$ one has $\langle F(z), v\rangle \leq c + \langle u, v\rangle$. From this one deduces that $F(H_0) \subset H_0+u = \{z+u : z\in H_0\}$. Indeed, if this is not the case then some $z\in \bd H_0 = L$ is mapped outside $H_0+u$, which implies that $\langle F(z)-u, v\rangle \geq c$ contradicting our previous remark. Thus $F(H_0)\subset H_0+u$ and since $u\in \Z^2$ one deduces inductively that $F^n(H_0) \subset H_0 + nu$.
	From this one deduces that 
	$$\Big\langle \frac{F^n(z)-z}{n}, v \Big\rangle \leq \langle u,v\rangle + \frac{c - \langle z, v\rangle}{n}.$$
	And conclusion follows easily using the definition of rotation set (noting that $F^n(z)-z$ is $\Z^2$-periodic).
\end{proof}

To prove (i), we consider for $c\in \R$ the line $L_c = \R\times \{c\}$. We claim that if $c=1/8$, then $F_{1,1}^2(L_c)$ lies strictly on the left of $L_{c+2}$. In other words $\pi_2(F_{1,1}^2(z))<c+2$ for all $z\in L_c$. To see this it suffices to compute explicitly $$\pi_2 (F_{1,1}^2(x, c)) - c = s(x) + s(x+s(c+s(x)))\leq 2,$$ since the maximum value of $s$ is $1$, and moreover equality may only hold if $x$ and $x+s(c+s(x))$ are both maxima of $s$, which means they are both equal to $1/4$ mod $1$. However $x=1/4$ mod $1$ implies $x+s(c+s(x)) = 1/4+s(c+1)$ mod $1$, and therefore this is only possible if $s(c+1) = 0$ mod $1$. In particular $c=1/8$ does not satisfy this property, and therefore $\pi_2 (F_{1,1}^2(x, c)) - c<2$ for all $x\in \R$. The proposition above applied to $F_{1,1}^2$, using $v = (0,1)$ and $u=(0,2)$, implies that $\rho(F_{1,1}^2)$ is contained in the half-plane $\{(x,y):y\leq 2\}$. 
Since $\rho(F_{1,1}^2) =2\rho(F_{1,1})$, it follows that $\max \pi_2(\rho(F_{1,1}))\leq 1$. Furthermore, since $x\mapsto\pi_2(F_{1,1}^2(x,c))-c$ is $1$-periodic, its maximum value is strictly smaller than $2$, and this remains valid if one perturbs $F_{1,1}$. Therefore $\max \pi_2(\rho(F_{\alpha,\beta}))\leq 1$ for all $(\alpha,\beta)$ close enough to $(1,1)$. The symmetries (\ref{eq:sym-axes}) and (\ref{eq:sym-R}) then imply that $\rho(F_{\alpha,\beta}) \subset [-1,1]\times [-1,1]$ if $(\alpha, \beta)$ is close enough to $(1,1)$. Finally, as observed at the begining of \S\ref{sec:apriori}, if $\alpha,\beta\geq 1$ then $\rho(F_{\alpha,\beta})$ contains the four vertices of this square, and therefore the rotation set is exactly the square $[-1,1]^2$, which proves (i).

In order to prove (ii), letting $F = F_{1/2,1/2}$, we verify that the line $L = \{(x,y) : x+y = 0\}$ is mapped strictly below $L+(2,0)$ by $F^4$. In other words, letting $\phi(x,y)=x+y$, one has $\phi(F^4(x,-x)) <2$ for all $x\in \R$. This can be verified by numerical (but certifiable) methods. Indeed, from the derivative (\ref{eq:jacob}) one may see that the Lipschitz constant of $F^4$ is less than $(\pi^2+2)^4 < 20000$, and a numerical computation (with error estimates) of $F^4(x,-x)$ for $x\in [0,1]$ with a step of $10^{-6}$ finds a maximum value of less than $1.95$ for $\phi(F^4(x,-x))$, therefore bounding the actual maximum by $20000\cdot 10^{-6} + 1.95 < 2$. The previous proposition with $v = (1,1)$, $c=0$, and $u=(2,0)$ implies that $\rho(F^4)\subset \{(x,y):x+y\leq 2\}$.  
Again, this persists under small perturbations of $F$, and using the fact that $\rho(F^4)=4\rho(F)$ we see that if $(\alpha,\beta)$ is close enough to $(1/2,1/2)$ one has $\rho(\Fab)\subset \{(x,y) : x+y \leq 1/2\}$, and by the symmetries of the rotation set we see that (close enough to $(1/2,1/2)$) the rotation set $\rho(\Fab)$ is contained in the square $Q$ described in (ii). Again, the remarks at the beginning of Section \ref{sec:apriori} imply that when $\alpha,\beta\geq 1/2$ the vertices of $Q$ belong to $\rho(\Fab)$, so (ii) follows.

\section{Continuity of the rotation set for Hamiltonian lifts of torus homeomorphisms} \label{RotationSetContinuity}

In this section we prove a general result which in particular implies the continuous dependence of the rotation set of $\Fab$ on the parameters
$(\alpha,\beta)$:
\begin{thm} \label{t.continuity}
The mapping
\[ 
 F\ \mapsto \ \rho(F)
\]
is continuous on $\hamlifts(\torus)$ (with respect
to the $\cC^0$-topology on $\hamlifts(\torus)$ and the Hausdorff
metric on the space of compact subsets of $\R^2$).
\end{thm}
\medskip

In order to prove Theorem~\ref{t.continuity}, let us first explain how existing
results can be used to reduce the problem to the case when
$F\in\hamlifts(\torus)$ has a rotation set of the form
\begin{equation}
  \label{e.rational_rotation_set}
  \rho(F) \ = \ \{tu\mid a\leq t \leq b\} \ , \quad \textrm{ where } u\in\Z^2
  \textrm{ and } a \leq 0 < b \ .
\end{equation}

Since the map $F\mapsto \rho(F)$ is upper-semicontinuous (Proposition \ref{p.rotset_semi}), when $\rho(F)$ is a singleton the continuity at $F$ follows. Likewise, when
$\rho(F)$ has nonempty interior, continuity at $F$ follows from
Proposition \ref{p.rotset_cont_int}.  Thus, we can assume that $\rho(F)$ is a nondegenerate line segment. Putting together Proposition \ref{p.fix-essential} and Theorem \ref{t.hamiltonian_LCT} we deduce that this segment must have rational slope. Since it also contains the origin (because $F$ is a Hamiltonian lift) it must be exactly of the form (\ref{e.rational_rotation_set}). \medskip

Hence, suppose from now on that $F\in\hamlifts(\torus)$ satisfies
(\ref{e.rational_rotation_set}). Fix $\eps>0$, and let $v,v'$ be two rational vectors (i.e., in $\Q^2$),
in $\rho(F)$ which are $\eps$-close to the endpoints. If we can show
that $v,v'\in\rho(G)$ for any sufficiently small perturbation $G$ of $F$ in
$\hamlifts(\torus)$, then by the convexity of the rotation set the whole segment
between $v$ and $v'$ will be contained in $\rho(G)$. This yields the lower
semicontinuity of $F\mapsto \rho(F)$, which together with the upper semicontinuity from Proposition \ref{p.rotset_semi} implies the continuity of this map. Therefore, it suffices to prove the following statement, which can be
applied to two pairs of rational points $w,v$ and $w',v'$ chosen close enough to
the endpoints of $\rho(F)$:

\begin{prop}\label{pr:maincontinuity}
Suppose that the rotation set of $F\in \hamlifts(\torus)$ is a line segment
containing $w\in \rho(F)\cap \Q^2\sm \{(0,0)\}$. Then, for each $v\in \Q^2$ lying in
the interval $I_w = \{tw \mid t\in (0,1)\}$, there exists a neighborhood $\mathcal
U$ of $F$ in $\hamlifts(\torus)$ such that every $G\in \mathcal U$ satisfies
$v\in \rho(G)$.
\end{prop}

\begin{proof}
Since $\rho(F^n)=n\rho(F)$, we can replace $F$ by an adequate power to assume
that $v$ and $w$ have integer coordinates. Note that since $n$ depends on (the
denominators of) $v$ and $w$, so does the neighbourhood $\cU$ chosen below.

Hence, we assume $v,w\in \Z^2$.  Since $F$ is a lift of an area-preserving
homeomorphism of $\T^2$ and its rotation set is a segment containing $w$, by Proposition \ref{p.reali}(3) there exists $z_0$ such that
$F(z_0) = z_0+w$.  Fix $0< \epsilon < 1/4$ such that $F(B_\epsilon(z_0))$ has
diameter smaller than $1/4$. Let $\mathcal U$ be a neighbourhood of $F$ in
$\hamlifts(\torus)$ with the property that every element $G\in \mathcal U$ is such that
$G(z_0)\in B_\epsilon(z_0+w)$ and $G(B_\epsilon(z_0))$ has diameter smaller than
$1/4$. We claim that $v\in \rho(G)$ for any such $G$.  To show this, we consider
an area-preserving homeomorphism $h$ defined on $B_\epsilon(z_0+w)$ which is the
identity in the boundary of the disk and such that $h(G(z_0)) = z_0+w$. We
extend $h$ to a homeomorphism $H\in \hamlifts(\torus)$ by $H(z) = h(z-\nu)+\nu$
if $z\in B_\epsilon(z_0+w)+\nu$ for some $\nu\in \Z^2$, and $H(z) = z$
otherwise. One easily verifies that $H\in \hamlifts(\torus)$, and $G' : = HG$
satisfies $G'(z_0) = z_0+w$.

Thus $w\in \rho(G')$, and since $G'$ is a Hamiltonian lift, we also have $(0,0)\in \rho(G')$. By convexity this implies $I_w\subset \rho(G')$. Moreover, since the mean rotation vector of Lebesgue measure for $G'$ is $(0,0)$, Proposition \ref{p.reali}(4) implies that there exists $z_1\in \R^2$ such that
$G'(z_1) = z_1 + v$. If $z_1\in B_\epsilon(z_0)$, then
$$\|G'(z_0)-G'(z_1)\| = \|(z_0+w) - (z_1-v)\| \geq \|w-v\| - \|z_0-z_1\| \geq
1-\epsilon > 3/4 \ ,$$ which contradicts the fact that $G'(B_\epsilon(z_0)) =
G(B_\epsilon(z_0))$ has diameter at most $1/4$. Thus $z_0\notin
B_\epsilon(z_0)$, and the same argument applied to its integer translations show
that $z_1\notin \bigcup_{u\in \Z^2} B_\epsilon(z_0)+u$. Since $G'$ coincides
with $G$ outside of this set, we conclude that $G(z_1) = G'(z_1) = z_1+v$. In
particular $v\in \rho(G)$ as claimed.
\end{proof}

\begin{rem}
  We note that the above argument can also be modified in order to give an
  alternative and elementary proof of \cite[Theorem
    B]{misiurewicz/ziemian:1991} (the continuity of $F\mapsto \rho(F)$ when $\rho(F)$ has nonempty interior). In order to show the persistence of a rotation
  vector $v\in\inte(\rho(F))\cap \Q^2$, it suffices to choose three rational
  vectors $w_1,w_2,w_3 \in \rho(F)$ such that $v$ is contained in the interior
  of their convex hull and to repeat the above construction simultaneously for
  three fixed points $z_1,z_2,z_3$ realising $w_1,w_2,w_3$, respectively.
\end{rem}

\section{The cusp along the diagonal} \label{Cusp}

In this section, we concentrate on parameters on or close to the diagonal, with
the aim to verify (in a qualitative way) the cusp form of the set $\cN$ in this
region. First, we note from (\ref{eq:sym-R}) that the rotation set of $\rho(F_{\alpha,\alpha})$ is invariant under the rotation by angle $\pi/2$, and therefore it cannot be a line segment of positive length. Hence, Corollary \ref{c.diagonal}, which states that the rotation set always has non-empty interior on the diagonal (excluding the origin) is an immediate consequence of Theorem~\ref{t.irrotational}.

\subsection{Absence of irrotational dynamics for $(\alpha,\beta)\neq 0$: Proof of Theorem~\ref{t.irrotational}}\label{sec:irrotational}

A torus homeomorhism $f\in\homeo_0(\torus)$ is called {\em irrotational} if it
has a lift $F$ that satisfies $\rho(F)=\{(0,0)\}$. In this case, we also say the lift $F$ is irrotational. The aim of
this section is to show that for the kicked Harper map this case can only occur
when $\alpha=\beta=0$, which is the statement of Theorem~\ref{t.irrotational}.

When $\alpha=0$ and $\beta\neq 0$ or vice versa, then it is obvious that
$\rho(\Fab)$ is a non-degenerate segment. Hence it remains to consider the case 
$\alpha\beta\neq 0$. For such parameters, as discussed in
Section~\ref{StandardFamily}, the fixed points of $\Fab$ are all elementary (i.~e. $1$ is not an eigenvalue of the derivative at these points). As $\Fab$ extends to a biholomorphic
mapping of $\C^2$, we know due to Ushiki's Theorem
\cite[p. 289]{HasselblattKatok2002Handbook} that $\Fab$ does not admit any
saddle connections between hyperbolic fixed points. Therefore,
Theorem~\ref{t.irrotational} is a consequence of the following more general
result on irrotational Hamiltonian torus homeomorphisms.
\begin{thm}\label{t.irrotatoinal_hamiltonians} Suppose that $f$ is a Hamiltonian torus diffeomorphism with a lift $F\in\hamlifts(\torus)$ such that $\rho(F) = \{(0,0)\}$.
Then $F$ either has a non-elementary fixed point, or it admits a saddle connection between hyperbolic fixed points. 
\end{thm}

\begin{proof}
Assume that every fixed point of $F$ is elementary. Since $\rho(F)=\{(0,0)\}$, every fixed point of $f$ is lifted to a fixed point of $F$, and since elementary fixed points are isolated, $f$ has finitely many fixed points. This implies that the fixed point set of $f$ is inessential (contained in some topological open disk in \torus), and Theorem \ref{t.hamiltonian_LCT} implies that $F$ has uniformly bounded displacements\footnote{in the special case where $F=F_{\alpha,\beta}$, this follows from Proposition \ref{p.diffusion_threshold}, and \cite{LeCalvezTal2015ForcingTheory} is not necessary.} , that is, we have 
\begin{equation}
\sup_{z\in \R^2, n\in \Z} \|F^n(z) - z\| \ < \ \infty \ .
\end{equation}
Let $H_0$ be the half-plane $\R\times (-\infty,0)$, and consider $H_1 =
\bigcup_{n\in \Z} F^n(H)$, which is $F$-invariant and $\Gamma_1$-invariant where
$\Gamma_1(x,y) = (x+1,y)$. Let $H$ be the union of $H_1$ with all bounded
connected components of the complement of $H_1$. Then $H$ is still $F$- and
$\Gamma_1$-invariant, bounded from above, and moreover it is a simply connected
open set. Let $U_0$ denote the projection of $H$ to the annulus $\A =
\kreis\times\R \simeq \R^2/\langle \Gamma_1\rangle$. The map $F$ induces a homeomorphism $\tilde F\colon \A\to \A$
which leaves $U_0$ invariant and commutes with the map $\tilde\Gamma_2\colon \A\to \A$
induced by the corresponding translation $\Gamma_2:(x,y)\mapsto (x,y+1)$ of $\R^2$.

We note that $\tilde F$ preserves some finite non-atomic measure $\mu$ of full
support. This can be seen by noting that the sets $A_k =
\Gamma_2^k(\Gamma_2(U_0)\sm U_0)$ are bounded, invariant, and $\tilde F$
preserves the measure $\mu_k$ given by the Lebesgue measure restricted to the
interior of $A_k$. Note that the boundary of each $A_k$ is a closed nowhere dense set, which implies that $\bigcup_{k\in\Z}\inte(A_k)$ is dense. Since each $\mu_k$ is finite (and $\mu_k(A_k)$ does not
depend on $k$), letting $\mu = \sum_{k\in \Z} 2^{-|k|}\mu_k$ we obtain a finite
$\tilde F$-invariant measure of full support.

Let $\mathbb{S}^2 = \A \cup\{+\infty, -\infty\}$ be the usual compactification
of $\A$ by topological ends (where $+\infty$ is the end on which $U$ does not
accumulate), and $U = U_0 \cup \{-\infty\}$, which is an open topological
disk. Extending $\tilde F$ (by fixing $\pm \infty$) we have an
orientation-preserving homeomorphism of $\mathbb{S}^2$ which leaves invariant
the open topological disk $U$.  The fact that the original map is irrotational
implies that the prime ends rotation number of $\tilde F$ in $U$ is $0$ (this
follows, for instance, from \cite[Props. 5.3-5.4]{franks/lecalvez:2003}, or more
directly from \cite{matsumoto2010rotation}).  Moreover, in a neighborhood of
$\bd U$, the fixed points of $\tilde F$ are elementary
(because, as elements of $\A$, they are projections of fixed points of
$F$). Since $\tilde F$ also preserves a finite measure of full support in
$\mathbb{S}^2$, it follows from Theorem 1.4 of
\cite{KLN2}
that $\bd U$
either contains a degenerate fixed point or consists entirely of hyperbolic
fixed points and saddle connections.
\end{proof}

\subsection{Pinching at the origin: Proof of Theorem~\ref{t.pinching}}
 We already know from Corollary \ref{c.diagonal} that the
set $\cN$ of parameters where the rotation set has nonempty interior includes
$\{(\alpha,\alpha) : \alpha\neq 0\}$, and therefore a neighborhood of the latter
set (since $\cN$ is open).  On the other hand, Figure~\ref{f.emptyint} suggests
that the set of parameters with nonempty interior (in the first quadrant) has
a cusp shape at the origin; that is, every line through the origin other than
the diagonal contains an interval around the origin where the rotation set has
empty interior.  This is confirmed by
Theorem~\ref{t.pinching}. We slightly reformulate the latter and prove the
following statement. Note that due to the symmetries described in
Section~\ref{StandardFamily}, it suffices to consider parameters below the
diagonal.

\begin{thm}\label{th:diagonal} For every $0\le\lambda<1$ there exists $\alpha_0=\alpha_0(\lambda)>0$
  such that for every $\alpha \in[0,\alpha_0(\lambda)]$ the rotation set of
  $F_{\alpha, \lambda \alpha}$ has empty interior.  Moreover, $\alpha_0 : [0,1)\to
    (0,+\infty)$ can be chosen continuous.
\end{thm}
\begin{rem}
 What we actually show is that if $\alpha$ is chosen smaller than
 $\alpha_0(\lambda)$, then there exist horizontal invariant closed curves
 (KAM curves) for $F_{\alpha,\lambda\alpha}$. This implies that the rotation set
 is contained in the horizontal axis.
\end{rem}

\begin{proof}[Proof of Theorem \ref{th:diagonal}]
Consider the vector field 
\[
W^{\lambda,\alpha}(x,y) \ = \ (s(y+\alpha\lambda s(x)),\lambda s(x)) \ ,
\]
and denote the corresponding flow by $\Phi^{\lambda,\alpha}$.  If we perform
Euler's method for the numerical integration of $W^{\lambda,\alpha}$, then the
map we obtain for time step $\alpha$ is exactly
\[
F_{\alpha,\lambda\alpha}(x,y)\ = \ (x,y)+\alpha W^{\lambda,\alpha}(x,y) \ .
\]
Let $n_\alpha=\lfloor 1/\alpha\rfloor$.  Although we have a dependence between
the vector field $W^{\lambda,\alpha}$ and the time step $\alpha$, standard
estimates on the convergence of the Euler method (as, for instance, provided by
Theorem~\ref{th:euler} in the appendix) imply that the $\cC^\infty$-distance
between $F_{\alpha,\lambda\alpha}^{n_\alpha}$ and the time-one map
$\Phi^{\lambda,\alpha}_1$ converges to zero as $\alpha\to0$. (Note here that
there exist uniform bounds for all derivatives of the vector fields
$W^{\lambda,\alpha}$ with $\alpha\in[0,1]$.) As at the same time
$\Phi^{\lambda,\alpha}_1$ clearly converges to $\Phi^{\lambda,0}_1$, this means
that for $\alpha$ sufficiently small the map
$F^{n_\alpha}_{\alpha,\lambda\alpha}$ is $\cC^\infty$-close to
$\Phi^{\lambda,0}_1$.

However, the flow $\Phi^\lambda=\Phi^{\lambda,0}$ with $\lambda\in[0,1]$ is easy
to analyse.  It is a conservative flow, which lifts a flow of $\T^2$ with two
hyperbolic singularities at $(0,0)$ and $(1/2,1/2)$ and two elliptic ones at
$(0,1/2)$ and $(1/2,0)$. When $\lambda = 1$, the hyperbolic singularities have
saddle connections as shown in Figure \ref{f.vectorfields}(a). When $\lambda<1$,
one may easily verify that these connections are replaced by homoclinic
connections as in Figure \ref{f.vectorfields}(b). The region complementary to
the elliptic islands on $\T^2$ consists of two essential horizontal annuli $A_1$
and $A_2$, and the dynamics on the $A_i$ is integrable, that is, all its orbits
are essential (horizontal) circles. Moreover, by the smoothness of the flow, and since each point is these annuli is periodic, the function assigning to each point its period is also smooth and constant on each invariant circle. Note that the set of invariant circles has a natural topology where it is homeomorphic to an open interval of $\R$. Finally, since the boundary of these annuli contain singularities, the function assigning to each circle the period of its point cannot be constant and thus must be strictly monotone in some sub-interval. Therefore one may find a smaller annulus
$A_0\subset A_1$ foliated by invariant circles such that $\Phi^\lambda_1$ is an
integrable twist map on $A_0$ 
By the KAM invariant curve theorem (see \S\ref{s.KAM}), any map sufficiently
$\cC^\infty$-close to $\Phi^\lambda_1$ will have horizontal invariant circles.  In
particular, if $\alpha$ is small enough, $F_{\alpha,\lambda\alpha}^{n_\alpha}$
has some horizontal invariant circle $C$, and therefore so has
$F_{\alpha,\lambda\alpha}$. Hence, $\rho(F_{\alpha,\lambda\alpha})$ must be
contained in a horizontal segment. This proves that for $\alpha$ small enough,
$\rho(F_{\alpha,\lambda\alpha})$ has empty interior.

Finally, we note that due to the stability of the KAM circles, we may choose
$\alpha_0$ such that it is uniformly bounded away from $0$ on any compact
subinterval of $[0,1)$. Reducing $\alpha_0$ further if necessary, we can
  therefore choose it continuously as a function $[0,1)\to (0,+\infty)$.
\end{proof}

\begin{figure}[h!]\label{f.vectorfields}
  \begin{center}
    \includegraphics[scale=0.5]{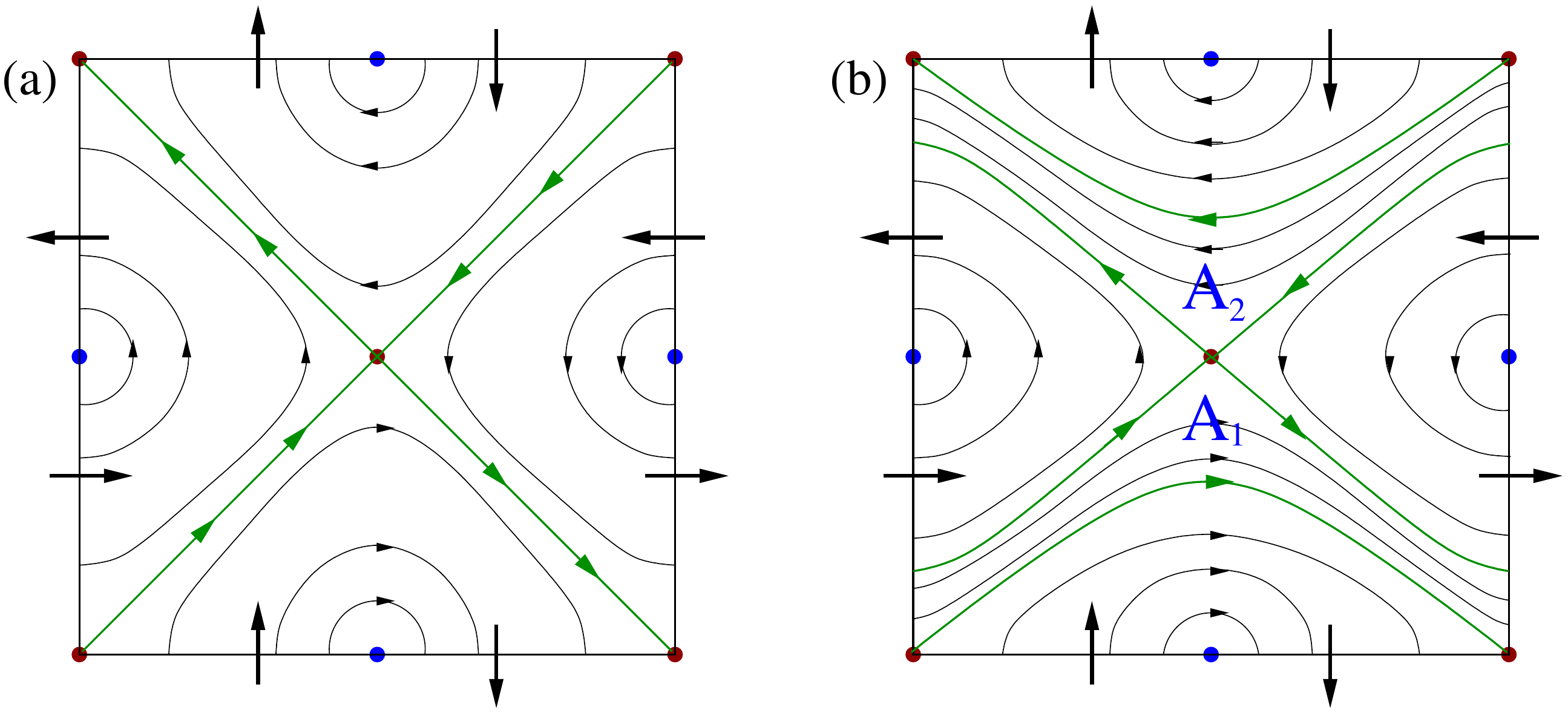}
  \end{center}
  \caption{Schematic picture of the (projections of the) vector fields
    $W^{\lambda,0}$ and the corresponding flows $\Phi^\lambda$ on the torus: (a)
    the case $\lambda=1$; (b) the case $\lambda<1$, with the two invariant annuli
    $A_1$ and $A_2$, bounded by homoclinic saddle-connections (in green).}
\end{figure}

\section{Large parameters: Proof of Theorem \ref{t.scaling}} \label{LargeParameters}

Recall that
\begin{eqnarray}
  \beta^-(\alpha) \ = \ \inf\{\beta>0\mid \inte(\rho(\Fab))\neq\emptyset\} \ ,
  \\ \beta^+(\alpha) \ = \ \sup\{\beta>0 \mid \inte(\rho(\Fab))=\emptyset\}\ ,
\end{eqnarray}
and Theorem~\ref{t.scaling} asserts that both these quantities are of order
$1/\sqrt{\alpha}$ for large $\alpha$, in the sense that there exists constants
$0<c<C$ such that
\begin{equation}
  \frac{c}{\sqrt{\alpha}} \ \leq \ \beta^-(\alpha)\ \leq \ \beta^+(\alpha)
  \ \leq \ \frac{C}{\sqrt{\alpha}} \ .
\end{equation}
We will treat the lower and upper estimate separately.

\begin{prop} \label{p.large_parameters_upper_bound}
  There exists $C>0$ such that $\beta^+(\alpha)\leq
  \frac{C}{\sqrt{\alpha}}$ for all $\alpha\geq 1/2$.
\end{prop}
We recall from \S\ref{sec:apriori} that $[-1/2,1/2]\ssq
\pi_1(\rho(\Fab))$ whenever $\alpha\geq 1/2$. Hence, in order to find an upper bound on $\beta^+(\alpha)$
it suffices to show that the rotation set is not contained in the horizontal
axis for any $\beta$ larger than the desired bound.

In order to do so, we will use a geometric argument that essentially relies on
the fact that the horizontal shift $H_\alpha$ induces a strong shear in most
parts of the phase space. As the proof does not use any specific properties of
the kicked Harper model and the construction may also be useful in other
situations, we work in a slighly more general setting. Consider homeomorphisms
of $\T^2$ of the following type: Let $varphi,\psi,:\R\to\R$ be two continuous
and 1-periodic functions. Let $H_{\varphi},V_{\psi}:\R^2\to\R^2$ be given by
\begin{eqnarray}
H_{\varphi}(x,y) & = & (x+\varphi(y),y)\ , \\ V_{\psi}(x,y)& = & (x, y+\psi(x))
\ ,
\end{eqnarray}
and define $F_{\varphi,\psi}= H_{\varphi}\circ V_{\psi}$ (note that, with
  this notation, the Harper map $F_{\alpha,\beta}$ should be denoted $F_{\alpha
    s,\beta s}$). Given a 1-periodic continuous function $\gamma:\R\to\R$, let
$$\mathrm{Var}_\gamma(\delta)= \min_{t\in \R}\left(\max_{x\in[t,
    t+\delta]}\gamma(x)-\min_{x\in[t, t+\delta]}\gamma(x)\right)$$ be the
minimal variation that the function $\gamma$ has on an interval of length
$\delta$.

\begin{prop}\label{pr:largeparametersgeneral}
Let $\varphi,\psi $ be such that $ \min_{x\in\R}\psi(x)\le
0<\max_{x\in\R}\psi(x)=\beta$, and such that there exists $\delta\le\beta/2$
such that $\mathrm{Var}_{\varphi}(\delta)\ge 2$. Then
$\pi_2\left(\rho(F_{\varphi,\psi})\right)\ge \beta-\delta.$
\end{prop}
\begin{proof}
Let $\alpha_0$ be the line segment joining $(0,0)$ to $(1,0)$. We will show by
induction that, for every $n\ge1$, there exists a curve $\alpha_n\subset
F_{\varphi,\psi}(\alpha_{n-1})$ such that
$\max\pi_1(\alpha_n)-\min\pi_1(\alpha_n)=1$ and $\alpha_n\subset \R\times
[n(\beta-\delta),n(\beta-\delta)+\delta]$. The last property clearly implies the
proposition, as it shows that there are points in $\alpha_0$ whose vertical
displacement after $n$ iterations is $\geq n(\beta-\delta)$.

Given $n\geq 1$, suppose that $\alpha_{n-1}$ satisfies the inductive assumption
(which is trivial for $\alpha_0$) and let
$a_{n-1}=\min_{z\in\alpha_n}\pi_1(z)$. There exists some $x_0,x_1\in
[a_{n-1},a_{n-1}+1]$ such that $\psi(x_0)=0$ and $\psi(x_1)=\beta$. Note that,
by the induction hypothesis, there exists $y_0,y_1$ such that both $(x_0, y_0)$
and $(x_1, y_1)$ belong to $\alpha_{n-1}$, and 
\[
\alpha_{n-1} \ \ssq \ [a_{n-1},a_{n-1}+1]\times
      [(n-1)(\beta-\delta),(n-1)(\beta-\delta)+\delta] \ .
\]
Note further that $V_\psi(x_0, y_0)=(x_0,y_0)$ and $\delta\le \beta/2$, so
\[
\pi_2(V_{\psi}(x_0,y_0))\ \le \ (n-1)(\beta-\delta)+\delta \le
 n(\beta-\delta) \ ,
\]
and $V_\psi(x_1, y_1) = (x_1, y_1+\beta)$, so
\[
\pi_2(V_{\psi}(x_1,y_1))\ \ge \ (n-1)(\beta-\delta)+\beta \ = \ 
n(\beta-\delta)+\delta \ .
\]
Moreover, $V_{\psi}(\alpha_{n-1})$ is still contained in the strip
$[a_{n-1},a_{n-1}+1]\times\R$. One deduces that there exists a sub-arc $\gamma$
of $V_{\psi}(\alpha_{n-1})$ contained in $[a_{n-1},a_{n-1}+1]\times
[n(\beta-\delta),n(\beta-\delta)+\delta]$ such that $\gamma$ intersects both the
upper and lower boundaries of this rectangle.

Now, as $\mathrm{Var}_{\varphi}(\delta)\ge 2$, we may find $y_0', y_1'$ in
$[n(\beta-\delta),n(\beta-\delta)+\delta]$ such that
$\varphi(y_1')-\varphi(y_0')\ge 2$. Let $x_0', x_1'$ be such that both $(x_0',
y_0')$ and $(x_1', y_1')$ belong to $\gamma$. Note that
$$\pi_1\left(H_{\varphi}((x_0', y_0'))\right)\le a_{n-1}+1+\varphi(y_0')\le
a_{n-1}+\varphi(y_1')-1,$$ and $$\pi_1\left(H_{\varphi}((x_1', y_1'))\right)\ge
a_{n-1}+\varphi(y_1').$$ Moreover, $H_{\varphi}(\gamma)$ is contained in the
strip $\R\times [n(\beta-\delta),n(\beta-\delta)+\delta]$. Choosing
$a_n=a_{n-1}+\varphi(y_1')-1$ one deduces the existence of a subarc $\alpha_n$
of $H_{\varphi}(\gamma)\subset F_{\varphi,\psi}(\alpha_{n-1})$ such that
\[
\alpha_n \ \ssq \ [a_n,a_n+1]\times [n(\beta-\delta),n(\beta-\delta)+\delta]
\]
and $\alpha_n$ intersects both the left and right boundaries of this
rectangle, 
 proving the induction
assumption for $n$ and thus the proposition.
 \end{proof}

\proof[\textbf{\textit{Proof of Proposition~\ref{p.large_parameters_upper_bound}}}.]
Let $s(x)=\sin(2\pi x)$ as before. First we observe that, if $0<\delta<1/2$, then $\mathrm{Var}_{s}(\delta) \geq \pi\delta^2$. This can be verified by noting that the interval $(x,x+\delta)$ contains a subinterval of the form $(y, y+\delta/2)$ where there is no critical point of $s$. Thus we may assume that $(y, y+\delta/2)\subset (-1/4, 1/4)$ (since $|\cos(2\pi t)|$ is $1/2$-periodic), and  $|s(y+\delta/2)-s(y)| = \int_{y}^{y+\delta/2} |2\pi \cos(2\pi t)|dt$. Using the bound $\cos(2\pi t) \geq 1-4|t|$ in $(-1/4, 1/4)$ one obtains $\int_{y}^{y+\delta/2} |2\pi \cos(2\pi t)|dt\geq  \pi\delta^2$. 

Note also that $\beta^+(\alpha) \leq 1/2$ if $\alpha\geq 1/2$ (see \S\ref{sec:apriori}). Since $\mathrm{Var}_{\alpha s}(\delta)=\alpha\mathrm{Var}_{s}(\delta)$, if $0<\beta < 1/2$ and
$\alpha \ge \frac{8}{\pi\beta^2}$, we get that $\mathrm{Var}_{\alpha
  s}(\beta/2)>2$. Taking $C=\frac{8}{\pi}$, we get by Proposition
\ref{pr:largeparametersgeneral} that if $\alpha\ge C/\beta^2$, then
$\rho(F_{\alpha,\beta})=\rho(F_{\alpha s,\beta s})$ is not contained in
$\R\times\{0\}$.
Hence, $(\alpha,\beta)\in\cN$ in this case,
thus proving that $\beta^+(\alpha)\leq C/\sqrt{\alpha}$ for all $\alpha\geq
1/2$.  \qed\medskip

\begin{prop} \label{p.standard_nontwist_rescaling}
  There exists a constant $c>0$ such that for any $\alpha\geq 1$ we have that
  $\beta^-(\alpha)\geq c/\sqrt{\alpha}$.
\end{prop}
\proof It will be convenient to consider the maps $G_{\alpha,\beta}=V_\beta\circ
H_\alpha$ instead of $F_{\alpha,\beta}$. Note that since
$G_{\alpha,\beta}=V_\beta\circ F_{\alpha,\beta}\circ V_{\beta}^{-1}$, we have
$\rho(F_{\alpha,\beta})=\rho(G_{\alpha,\beta})$ and may therefore replace
$F_{\alpha,\beta}$ by $G_{\alpha,\beta}$ in the definition of $\beta^-(\alpha)$
in \ref{t.scaling}. Further for any $\alpha_0>0$ the restriction of
$F_{\alpha_0,0}$ to $\R\times[-1/8,1/8]$ is a lift of the completely integrable
twist map $f_{\alpha_0,0}$ on the annulus $A$ obtained by projecting the
corresponding strip $\R\times [-1/8, 1/8]$. By the KAM theorem (Theorem \ref{t.KAM}), $f_{\alpha_0,0}$
has stable KAM circles, and thus $f_{\alpha,\beta}$ has a horizontal KAM circle
whenever $(\alpha,\beta)$ is close enough to $(\alpha_0,0)$. By a compactness
argument, this guarantees that for each $M>0$ there is a constant $c_M$ such
that $\beta^-(\alpha) > c_M > 0$ whenever $\alpha\in [1,M]$. As a consequence,
it will be sufficient to prove the estimate of the lemma for large enough values
of $\alpha$.

Let $\A=\T^1\times\R$ and $\kappa=4\pi^2=|s''(1/4)|$ and consider the parameter
family of annular diffeomorphisms $S_{\alpha,\beta}\colon \A\to \A$
lifted by
\[
\tilde S_{\alpha,\beta} : \R^2\to \R^2 \quad , \quad (x,y) \ \mapsto \ V_{\beta}(x+\alpha -\kappa
y^2,y) \ .
\]
We note that this family is sometimes referred to as the {\em standard non-twist
  map}
(e.g.\ \cite{ShinoharaAizawa1998NontwistMapsTransitionToChaos,ShinoharaAizawa1998NontwistMapsTransitionToChaos}).
For each $\eta>0$ and $\alpha_0\in\R$, the restriction of the map
$S_{\alpha_0,0}$ to $\T^1\times[\eta,1]$ is a completely integrable twist
map and therefore has stable horizontal KAM circles. Moreover, we have
$S_{\alpha_0+1,\beta}=S_{\alpha_0,\beta}$, so that $\alpha_0$ can be viewed as
an element of \kreis. Hence, by compactness we obtain that there exist constants
$b,\eps_0>0$ and $k_0\in\N$ such that any smooth injective map $G: \A\to\A$
whose restriction to $\cA=\kreis\times[0,1]$ is $\eps_0$-close to
$S_{\alpha_0,0|\cA}$ in the $\cC^{k_0}$-metric for some $\alpha_0\in\R$ has
horizontal KAM circles.

Now, given $\alpha,\beta\in\R$, consider the rescaling
\[
\tilde G_{\alpha,\beta} \ = \ \Phi_\alpha\circ G_{\alpha,\beta}\circ
\Phi_{\alpha}^{-1}
\]
of $G_{\alpha,\beta}$, where
$\Phi_{\alpha}(x,y)=(x,\sqrt{\kappa\alpha}(y-1/4))$. Note that
$\Phi_{\alpha}\circ V_{\beta}= V_{\sqrt{\kappa\alpha}\beta}\circ \Phi_{\alpha}$,
and therefore
\[
\tilde G_{\alpha,\beta} \ = \ V_{\sqrt{\kappa\alpha}\beta}\circ
\Phi_{\alpha}\circ H_{\alpha}\circ \Phi_{\alpha}^{-1} \ =
V_{\sqrt{\kappa\alpha}\beta}\circ \tilde G_{\alpha,0} \ .
\]
Let $\widehat G_{\alpha,\beta}:\A \to \A$ be the homeomorphism naturally induced
by $\tilde G_{\alpha,\beta}$ on $\A$. Then it can be checked that, due to the
above rescaling, the maps $\widehat G_{\alpha_0+n,0|\cA}$ converge to
$S_{\alpha_0,0|\cA}$ as $n\to\infty$ in the $\cC^k$-topology for any
$k\in\N$. (Note here that $\widehat G_{\alpha_0+n,0}\neq \widehat
G_{\alpha_0,0}$ for $n\in\N\smin\{0\}$, since the rescaling that is carried out
before projecting to $\cA$ is different for the two maps.) Moreover, the
convergence is uniform in $\alpha_0\in[0,1]$. Hence, there exists a constant
$M>0$ such that such that for any $\alpha>M$ the map $\widehat G_{\alpha,0|\cA}$
is $\eps_0/2$-close to $S_{\alpha,0|\cA}$ in the $\cC^{k_0}$-topology.

Further, there exists $\delta>0$ such that for any $\alpha\in \R$ and
$\tilde\beta\in(0,\delta)$ the map $\tilde
G_{\alpha,\tilde\beta/\sqrt{\kappa\alpha}}=V_{\tilde\beta}\circ \tilde
G_{\alpha,0}$ is $\eps_0/2$-close to $\tilde G_{\alpha,0}$ in the
$\cC^{k_0}$-topology. As a consequence, we obtain that for any $\beta\in
(0,\delta/\sqrt{\kappa\alpha})$ the map $\widehat G_{\alpha,\beta|\cA}$ is
$\eps_0/2$-close to $\widehat G_{\alpha,0|\cA}$, and thus $\eps_0$-close to
$S_{\alpha,0|\cA}$ in the $\cC^{k_0}$-topology when $\alpha>M$. By the above,
this means that $\widehat G_{\alpha,\beta}$ has invariant KAM circles. However,
as $\widehat G_{\alpha,\beta}$ is just a rescaling of the projection of
$G_{\alpha,\beta}$ to $\A$, this means that the rotation set of
$G_{\alpha,\beta}$ is confined to the horizontal axis, that is,
$\rho(G_{\alpha,\beta})\ssq \R\times\{0\}$.  Letting $c =
\min\{\delta/\sqrt{\kappa},\, c_M\}$ (where $c_M$ is the constant from the
beginning of the proof) we conclude that $\beta^-(\alpha)\geq c/\sqrt{\alpha}$
for all $\alpha\geq 1$, as required.  \qed\medskip

\section{Questions and final remarks}\label{questionsandremarks}

The kicked Harper map, by which we mean the whole parameter family
$(f_{\alpha,\beta})_{\alpha,\beta\in\R}$, shows a rich variety of different
dynamical behaviours and phenomena. We believe that its study as a paradigmatic
example of smooth torus dynamics can be extremely fruitful and may lead to
general insights about torus dynamics and rotation theory on surfaces that go
well beyond the context of this particular example. With the results presented
above, we have merely scratched at the surface of a multitude of intriguing open
problems that can be investigated in this context. In the remainder of this
section, we collect a few directions in which future research on this topic may
be oriented.

\subsection{Structure of the parameter regions} 

The aim for a better understanding of the structure of the parameter regions
$\cE$ and $\cN$ leads to a number of further questions concerning their
qualitative and quantitative properties. First of all, in analogy to the
well-known problems in the study of Julia sets in complex dynamics, one may ask
\begin{itemize}
\item Are the sets $\cE$ and $\cN$ connected? Are they locally connected?
\end{itemize}
As Figure \ref{fig:zoom_duplo} shows, even connectedness should not be taken for granted. 

\begin{figure}
\begin{center}
  \includegraphics[scale=0.2]{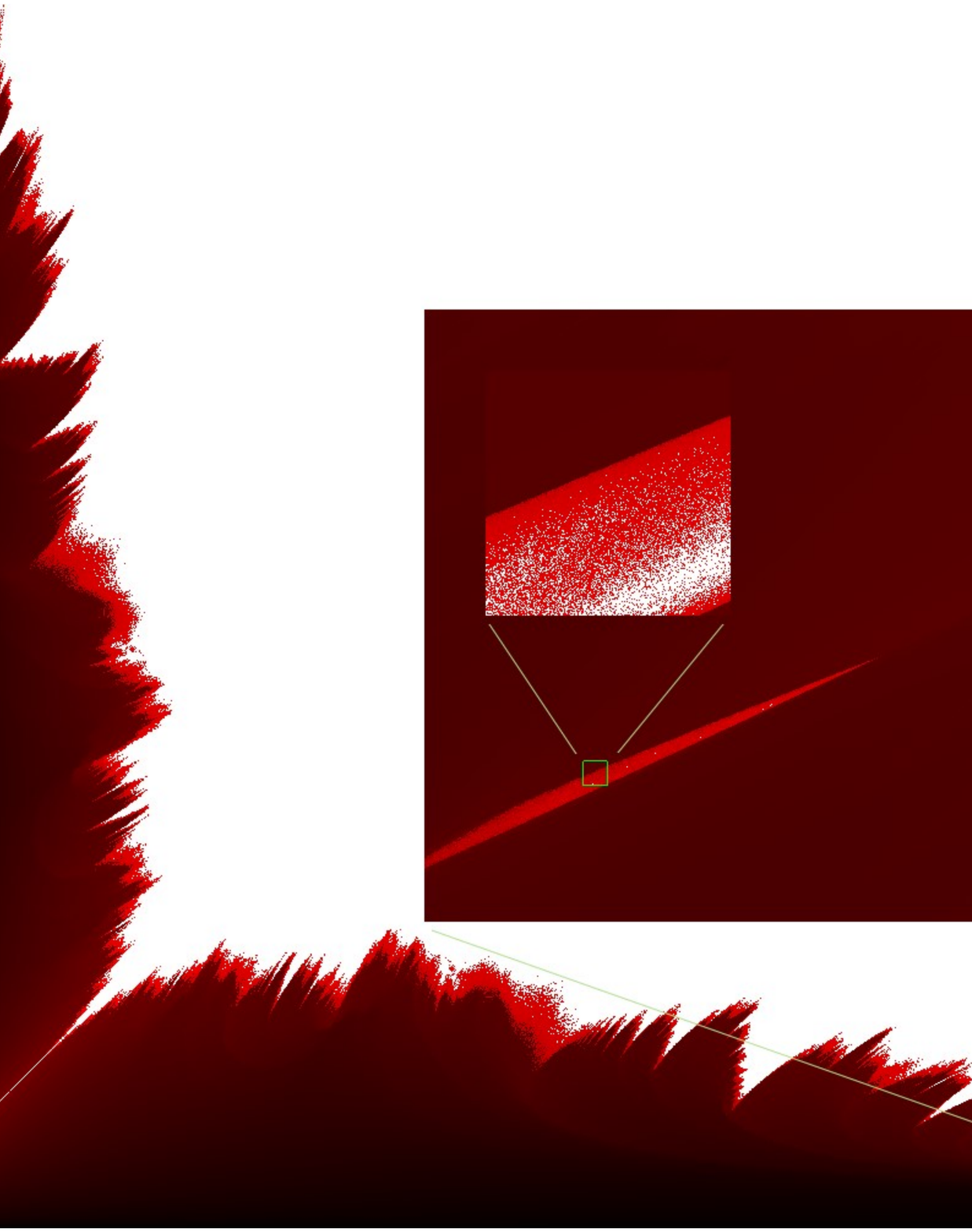}
  \caption{More detailed look at the sets $\cE$ and $\cN$ with zoomed regions where numerical estimates indicate the possibility of non-connectedness of the sets. The first zoomed square corresponds approximately to the parameter region $[0.915, 0.935]\times [0.065, 0.085]$. The second zoomed square corresponds approximately to the parameter region
$[0.920, 0.921]\times [0.069, 0.070]$.}  \label{fig:zoom_duplo}
\end{center}
  \end{figure}

Another problem that we leave open here is that of the seemingly periodic
structure of the set $\cE$ observed in Figure~\ref{f.emptyint_along_axes} (and
described previously in \cite{Shinohara2002DiffusionThreshold}). In mathematical
terms, one may formulate it as follows.
Let $M_2((x,y),t)= (x, ty)$
\begin{conj}
  The sequence $A_n=M_2((\cE\cap [n,n+1]\times[0,1])-(n,0), \sqrt{n})$ 
 converges in Hausdorff
  distance to the set 
\[
A \ = \ \{(\alpha,\beta)\in[0,1]^2\mid S_{\alpha,\beta} \textrm{ admits
  unbounded orbits}\} ,
\]
where $S_{\alpha,\beta}$ is the standard non-twist map introduced in the proof
of Proposition~\ref{p.standard_nontwist_rescaling}. 
\end{conj}
 Various further questions may be asked about the tongue structure that appears
 in Figures~\ref{f.emptyint} and \ref{f.emptyint_along_axes}. On a heuristic
 level, it seems plausible that the tongues of the region $\cN$ that reach into
 the region $\cE$ should somehow correspond to `resonances' appearing in the
 dynamics that make it easier to break all KAM circles, so that diffusion can
 take place. This should correspond to the appearence and disappearence of
 certain periodic orbits. However, the precise mechanisms are not at all clear
 to us. We refer to
 \cite{HowardHohs1984Reconnection,Shinohara2002DiffusionThreshold,Leboeuf1998KickedHarper,LeboeufKurchanFeingoldArovas1990PhaseSpaceLocalization,Zaslavsky2007PhysicsOfChaos}
 for more details and some phenomenological descriptions.

\subsection{Monotonicity properties}

Another aspect that is not well-understood and prompts a multitude of questions
is that of the dependence of the rotation set on the parameters.  Apart from the
continuity derived in Section~\ref{RotationSetContinuity}, little is
known. Specifically, one may ask about monotonicity properties: when do $0\leq
\alpha\leq \tilde\alpha$ and $0\leq\beta\leq\tilde\beta$ imply
$\rho(F_{\alpha,\beta})\ssq\rho(F_{\tilde\alpha,\tilde\beta})$. For instance, we
have a natural upper bound $\rho(F_{\alpha,\beta})\ssq
[-\alpha,\alpha]\times[-\beta,\beta]$ on the rotation set. However, while this
upper bound grows monotonically with the parameters, the same is not true in
general for the rotation set itself.

One way to see this is to consider parameters $\alpha=0$ and $\beta\in
(0,1)$. In this case, we have
$\rho(F_{0,\beta})=\{0\}\times[-\beta,\beta]$. However, for any parameter pair
$(\alpha,\beta)$ an average vertical displacement of $\beta$ is only possible if
an orbit stays exactly on the vertical line $\{1/4\}\times\R$, or converges to
it. This is not possible for $\alpha\in(0,1)$, so that $(0,\beta)\notin
\rho(F_{\alpha,\beta})$ in this case, and therefore $\rho(F_{0,\beta})
\nsubseteq \rho(F_{\alpha,\beta})$.

In contrast to this, numerical simulations based on
\cite{JaegerPadbergPolotzek2017RotationSets} suggest that the rotation set
behaves monotonically along the diagonal.

\begin{conj}
  If $0\leq \alpha \leq \tilde\alpha$, then
  $\rho(F_{\alpha,\alpha})\ssq\rho(F_{\tilde \alpha,\tilde \alpha})$.
\end{conj}

\subsection{Mode-locking}

A well-known and -studied phenomenon in the context of rotation theory is that
of mode-locking, which refers to the stability of rotation numbers, vectors or sets
under perturbations of the system. In the context of torus dynamics, it was
shown in \cite{Passeggi2013RationalRotationSets} that there exists an open and
dense subset of $\homeo_0(\torus)$ on which the rotation set is locally constant
and a rational polygon, and in \cite{GuiheneufKoropecki2016RotationSetStability}
the same statement was shown to hold when restricted to
${\mathrm{Homeo}}^\mathrm{ap}_0(\torus)$.  However, it is not clear if the
analogous statement is still true if one restricts to the parameter family $(f_{\alpha,\beta})_{\alpha,\beta\in\R}$, although the recent results in \cite{LeCalvezZanata2015RationalModeLocking}, showing that for any analytic one parameter family $G_t\in {\mathrm{Homeo}}^\mathrm{ap}_0(\torus)$ the rotation 
set cannot strictly increase over a whole interval $I\ssq \R$ (that is, 
$\rho(G_s)\subset \inte(\rho(G_t))$ cannot hold for all $s<t$ in $I$), point in that direction.

So, the following questions are open. 
\begin{itemize}
\item Is it true that there exists an open and dense set $M\ssq\R^2$ such that
  the mapping $(\alpha,\beta)\mapsto \rho(F_{\alpha,\beta})$ is locally constant
  on $M$?
\item Is it true that whenever the mapping
  $(\alpha,\beta)\mapsto\rho(F_{\alpha,\beta})$ is locally constant, the
  rotation set is a rational polygon?
\item Is there an open and dense set $A\ssq \R$ such that the mapping
  $\alpha\mapsto \rho(F_{\alpha,\alpha})$ is locally constant and only has
  rational polygons as images on $A$.
\end{itemize}

In this context, we note that the numerical computation or approximation of
rotation sets is an intricate problem in itself, such that it is difficult to
obtain numerical evidence concerning the occurrence or density of
mode-locking. We refer to \cite{JaegerPadbergPolotzek2017RotationSets} for
details on the numerical aspects. Using the algorithm developed there, it is
possible to identify some specific mode-locked regions in the kicked Harper
family, for instance around parameters $(\alpha,\beta)=(0.66,0.66)$. A particular case where it is possible to detect mode-locking is at $(0.5,0.5)$; near these values, $\rho(F_{\alpha,\beta})$ is the square with vertices $(-1/2,0)$, $(0, -1/2)$, $(1/2,0)$, $(0, 1/2)$, as we mentioned in \S\ref{sec:apriori}.

\subsection{Shape of rotation sets}

Another general open problem in torus dynamics is that of the possible shapes of
rotation sets. Due to \cite{MisiurewiczZiemian1989RotationSets}, it is known that
the rotation set of a torus homeomorphism is always convex, and Kwapisz showed
that every rational polygon (a polygon with all vertices in $\Q^2$) are
realised. Moreover, a few examples of non-polygonal rotation sets have been
described \cite{kwapisz:1995,BoylandDeCarvalhoHall2016NewRotationSets}, but all
of these only have a countable number of extremal points. Hence, it is
completely open if a set like the unit disk may appear as the rotation set of a
torus homeomorphism.

\begin{itemize}
\item Which sets do appear as rotation sets in the family
  $(F_{\alpha,\beta})_{\alpha,\beta\in\R}$?
\end{itemize}

\subsection{Phase space}

Finally, questions that are typically studied in the context of the Chirikov-Taylor standard family (of area-preserving twist maps) may equally be asked for the
kicked Harper model.
\begin{itemize}
\item Do elliptic islands exist for all/Lebesgue-almost all parameters
  $\alpha,\beta\neq 0$.
\item Conversely, are there parameters for which the kicked Harper map is
  topologically transitive/ergodic with respect to Lebesgue?
\item What is the Lebesgue measure of the complement of the union of all
  KAM circles/elliptic islands?
\end{itemize}

\subsection{Transverse foliation}

A number of recent advances in surface dynamics have been based on the concept
of transverse foliations (Brouwer-Le Calvez foliations) and a related forcing
theory developed in \cite{LeCalvezTal2015ForcingTheory}. As we have not made use
of this theory, we refrain from going into more detail here. However, for
readers that are familiar with the topic, we want to point out that the
existence of a transverse foliation (which in general follows from the work of
Le Calvez in \cite{LeCalvez2005FoliatedBrouwer}) can be seen quite easily for
the kicked Harper model. If one considers the homotopy $(h^t)_{t\in[0,1]}$ between the
identity and $f_{\alpha,\beta}$ lifted by 
\[
 (x,y) \mapsto H^t(x,y) = 
\begin{cases}
  (x,y+2t\beta s(x)), & 0\leq t\leq 1/2\\ (x+(2t-1)\alpha s(y+\beta s(x)),y+\beta
  s(x)) \ , & 1/2< t\leq 1
\end{cases}
\]
then each path of this isotopy that connects a point $(x,y)$ to its image under
$f_{\alpha,\beta}$ consists of a vertical and a horizontal segment (possibly
degenerate). Moreover, the orientation of these segments is only determined by
the quadrant of $\torus$ in which the segment starts. This allows to see that
the oriented foliation shown in Figure~\ref{f.foliation} is positively
transverse to the dynamics, that is, the paths of the homotopy the leaves of the
foliation in a transverse way from left to right, for all parameters
$\alpha,\beta\neq 0$ at the same time. The four common fixed points
$(0,0),(1/2,0),(0,1/2),(1,1)$ of the kicked Harper maps are singularities of the
foliation.
\begin{figure}
\begin{center}
  \includegraphics[scale=0.5]{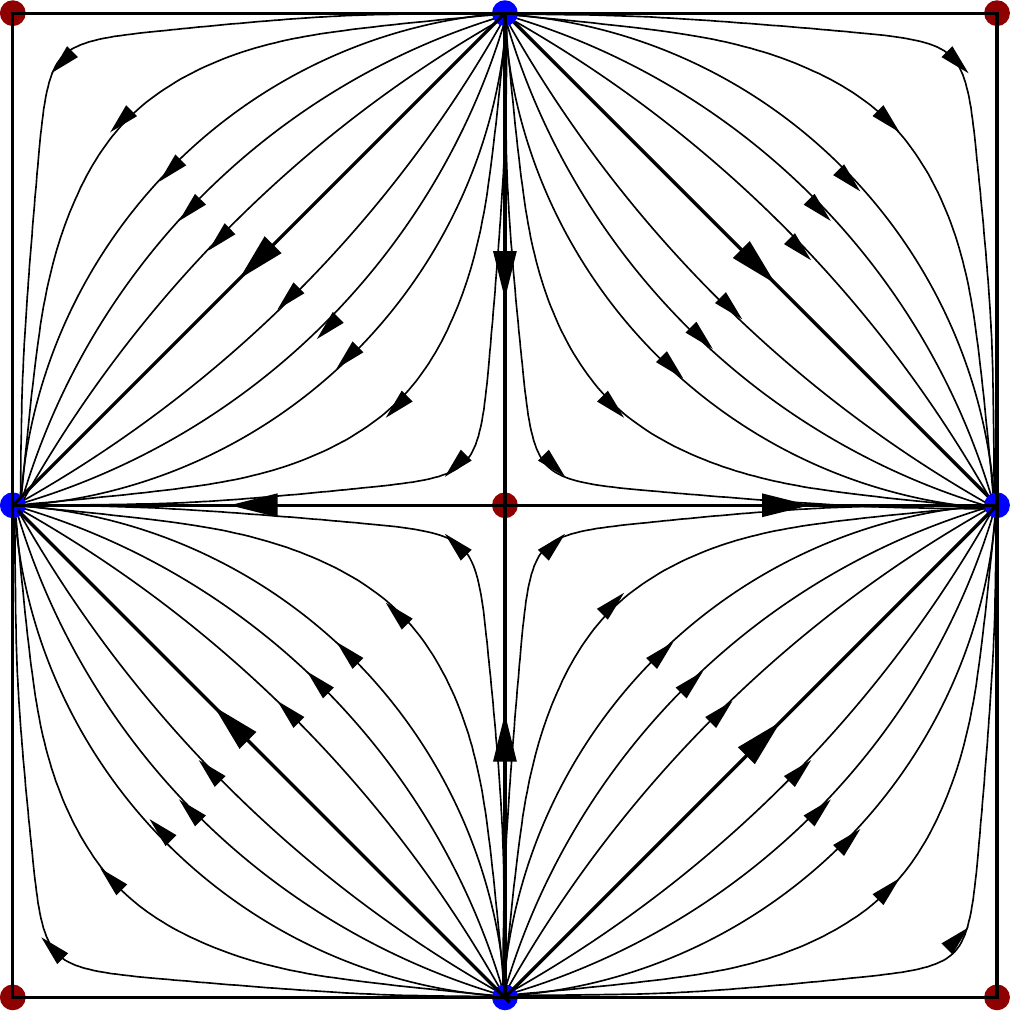}
  \caption{\label{f.foliation}Transverse foliation for the kicked Harper map with $\alpha,\beta\neq 0$.}
\end{center}
  \end{figure}

\section*{Appendix}
\addtocounter{section}{1}

\begin{thm}[$\cC^r$-convergence of Euler's method]\label{th:euler}
  For each $r\geq 1$ and $M>0$ there exists $C_r = C_r(M) > 0$ such that the following property holds.
Let $z\mapsto V(z)$ be a $\cC^{r+1}$ vector field $(r\geq 1)$ in $\R^n$, and
$z_0\in \R^n$ a point such that the corresponding flow $\phi^t(z_0)$ is defined for
all $t\in [0,1]$, and assume that the $\cC^{r+1}$ norm of $V(z)$ is at most $M$
for all $z$ in the $\epsilon$-neighborhood $U_\epsilon$ of $\{\phi^t(z_0): t\in
[0,1]\}$. Then the function $G_\delta(z) = z + \delta V(z)$ satisfies
$$\|D^rG_\delta^{n}(z_0)-D^r\phi^{n\delta}(z_0)\|\leq C_r\delta$$
for all $0<\delta<\min\{1,\epsilon\}/C_r$ and $n\leq\floor{1/\delta}$.
\end{thm}

\begin{proof}[Sketch of the proof]
Without loss of generality we assume $\epsilon<1$. Denote by $M_r$ the $\cC^r$
norm of $V$ in $U_{\epsilon}$.  Iterating $G_\delta$ produces an Euler
approximation of the solutions of $z' = V(z)$, and we have the following
well-known estimate for the error in Euler's method:
 $$\|\phi^{n\delta}(z_0) - G_\delta^n(z_0)\| \leq \delta M_0(e^{M_1\delta(n+1)}
 - 1),$$ which holds for all $n$ such that the right hand side is smaller than
 $\epsilon$. In particular, if $C_0 = M_0(e^{2M_1} - 1)$ then
$$\|\phi^{n\delta}(z_0) - G_\delta^n(z_0)\| \leq C_0 \delta$$
holds whenever $\delta<\epsilon/C_0$ and $n\leq n_\delta$.
Thus the claim holds for $r=0$.

To get a similar estimate for the derivatives, we will use the previous
observations in a new vector field. To avoid cumbersome notation with higher
order derivatives, we omit details about the spaces to which each object
belongs; this should be clear from context.

We will use the following notation: $D_*^kV(z) = (DV(z), D^2V(z), \dots,
D^kV(z))$.  Let $\Gamma_k$ be a $\cC^\infty$ map such that if $f,g\colon \R^n\to
\R^n$ are two $\cC^k$ maps and $h=f\circ g$, 
$$\Gamma_k(D_*^kf(g(z)), D_*^kg(z)) =D^k h(z).$$
An explicit formula for $\Gamma_k$ can be given (for instance Faa di Bruno's formula).

Let $u = (z, u_1, \dots, u_r)$, $W_0(u) = V(z)$,
$$W_k(u) = \Gamma_k(D_*^kV(z), u_1, \dots, u_k)$$
and 
$$W(u) = (W_0(z), W_1(z, u_1), \dots, W_r(z, u_1, \dots, u_r)).$$

Then it is easy to verify that the solution to
\begin{equation}\label{eq:ODEr}
u' = W(u)
\end{equation}
with initial condition $u(0) = v(z) := (z, I, 0, \dots, 0)$ is 
\begin{equation}
\label{eq:phir}
\phi^t_r(v(z)) = (\phi^t(z), D\phi^t(z), \dots, D^r(\phi^t(z)).
\end{equation}

Let $G_{r,\delta}(u) = u + \delta W(u)$ be the Euler approximation of the flow
given by the vector field $W$. If $U_\epsilon^r$ denotes the
$\epsilon$-neighborhood of $\{\phi^t_r(v(z_0)) : t\in [0,1]\}$, we then know
from the case $r=0$ that there exists $C_r>0$ such that whenever
$\delta<\epsilon/C_r$ and $n\leq n_\delta$,
\begin{equation}\label{eq:Grconv}
\norm{\phi^{n\delta}_r(v(z_0)) - G^n_{r, \delta}(v(z_0))} \leq C_r\delta.
\end{equation}
where $C_r$ depends only on the $\cC^1$ norm of $W$ in $U_\epsilon^r$. 

For $t\in [0,1]$ and $1\leq k\leq r$, there is a uniform bound
$\norm{D^k\phi^{t}_r(v(z_0))}\leq K$ depending only on $M_r$. This can be seen
noting (for instance from Faa di Bruno's formula) that
$$\Gamma_k(D_*^kV(z), (u_1, \dots, u_k)) = \Lambda_k(D_*^kV(z), u_1, \dots,
u_{k-1}) + DV(z)u_k,$$ where $\Lambda_k$ is another (explicit) function, and
applying Gronwall's inequality for each coordinate $u_k$ in (\ref{eq:ODEr})
inductively. We leave these details to the reader.

This implies that any $u\in U_{\epsilon}^r$ satisfies $\norm{u_i}\leq K+\epsilon
\leq K+1$ for $1\leq i\leq r$. Using this fact and the explicit form of $W$ we
see that the $\cC^1$ norm of $W$ in $U_\epsilon^r$ is bounded by a constant
depending only on $M_{r+1}$. In particular the constant $C_r$ above depends only
on $M_{r+1}$.

Since the $r$-th coordinate of $\phi^{n\delta}_r(v(z_0))$ is
$D^r\phi^{n\delta}(z_0)$, in view of (\ref{eq:Grconv}) and (\ref{eq:phir}), to
complete the proof it suffices to show that
\begin{equation}
 G_{r, \delta}^n(v(z)) = (G_\delta^n(z), DG_\delta^n(z), \dots, D^rG_\delta^n(z)).
\end{equation}

This clearly holds when $n=0$ due to the defintion of $v(z_0)$; and for $n\geq 0$ 
$$G_{r,\delta}^{n+1}(v(z)) = G_{r,\delta}(G_{r,\delta}^n(v(z))) =
G_{r,\delta}^n(v(z)) + \delta W(G_{r,\delta}^n(v(z))).$$ So assuming by
induction that the claim holds for $n$, looking at the $k$-th coordinates we get
$$G_{r,\delta}^{n+1}(v(z))_k = D^kG_\delta^n(z) + \delta\Gamma_k(G_\delta^n(z),
DG_\delta^n(z), \dots D^kG_\delta^n(z)).$$ Using the definition of $\Gamma_k$,
this is equal to
$$D^kG_\delta^n(z) + \delta D^kV(G_\delta^n(z)) = D^k(G_\delta(G_\delta^n(z))) = D^k G_\delta^{n+1}(z),$$
which proves the induction step. This completes the proof.

\end{proof}

\bibliography{dynamics} \bibliographystyle{alpha}

\end{document}